\documentclass[11pt,a4paper,reqno]{amsart}
\usepackage[english]{babel}
\usepackage[applemac]{inputenc}
\usepackage[T1]{fontenc}
\usepackage{palatino}
\usepackage{verbatim}
\usepackage{amsmath}
\usepackage{amssymb}
\usepackage{amsthm}
\usepackage{amsfonts}
\usepackage{graphicx}
\usepackage{mathtools}
\usepackage{enumitem}

\DeclarePairedDelimiter\floor{\lfloor}{\rfloor}
\usepackage[colorlinks = true, citecolor = black]{hyperref}
\usepackage{tikz}
\usetikzlibrary{arrows}

\newcommand{\R}{\mathbb{R}}

\newcommand{\N}{\mathbb{N}}

\newcommand{\calJ}{\mathcal{J}}

\newcommand{\calL}{\mathcal{L}}
\newcommand{\calD}{\mathcal{D}}
\newcommand{\calH}{\mathcal{H}}

\newcommand{\calB}{\mathcal{B}}

\newcommand{\calI}{\mathcal{I}}
\newcommand{\calS}{\mathcal{S}}

\newcommand{\calR}{\mathcal{R}}

\newcommand{\spt}{\operatorname{spt}}
\newcommand{\Hd}{\dim_{\mathrm{H}}}

\newcommand{\spa}{\operatorname{span}}
\newcommand{\diam}{\operatorname{diam}}
\newcommand{\card}{\operatorname{card}}
\newcommand{\dist}{\operatorname{dist}}

\newcommand{\sgn}{\operatorname{sgn}}

\newcommand{\calC}{\mathcal{C}}
\newcommand{\calW}{\mathcal{W}}

\newcommand{\dd}{\,\mathrm{d}}
\newcommand{\m}{\mathfrak{m}}
\newcommand{\n}{\mathfrak{n}}

\numberwithin{equation}{section}

\theoremstyle{plain}
\newtheorem{thm}{Theorem}[section]
\newtheorem{conjecture}[thm]{Conjecture}
\newtheorem{lemma}[thm]{Lemma}

\theoremstyle{definition}

\newtheorem{definition}[thm]{Definition}
\newtheorem{assumption}[thm]{Assumption}

\theoremstyle{remark}
\newtheorem{remark}[thm]{Remark}

\newcommand{\nref}[1]{(\hyperref[#1]{#1})}

\begin{document}

\pagestyle{headings}

\title{A Marstrand-type restricted projection theorem in $\R^{3}$}

\author{Antti K\"aenm\"aki}
\address{Research Unit of Mathematical Sciences, University of Oulu, P.O. Box 8000, FI-90014, University of Oulu, Finland}
\email{antti.kaenmaki@oulu.fi}

\author{Tuomas Orponen}
\address{Department of Mathematics and Statistics\\ University of Jyv\"askyl\"a,
P.O. Box 35 (MaD)\\
FI-40014 University of Jyv\"askyl\"a\\
Finland}
\email{tuomas.t.orponen@jyu.fi}

\author{Laura Venieri}

\email{laura.venieri.lv@gmail.com}

\date{today}

\thanks{A.K. was partially supported by the ERC grant 306494. T.O. was supported by the Academy of Finland via grants 274512 and 309365. L.V. was supported by the Vilho, Yrj\"o ja Kalle V\"ais\"al\"a Foundation. All authors were supported by the Finnish Center of Excellence in Analysis and Dynamics Research.}
\date{\today}

\begin{abstract} Marstrand's projection theorem from $1954$ states that if $K \subset \R^{3}$ is an analytic set, then, for $\calH^{2}$ almost every $e \in S^{2}$, the orthogonal projection $\pi_{e}(K)$ of $K$ to the line spanned by $e$ has Hausdorff dimension $\min\{\Hd K,1\}$. This paper contains the following sharper version of Marstrand's theorem. Let $V \subset \R^{3}$ be any $2$-plane, which is not a subspace. Then, for $\calH^{1}$ almost every $e \in S^{2} \cap V$, the projection $\pi_{e}(K)$ has Hausdorff dimension $\min\{\Hd K,1\}$. For $0 \leq t < \Hd K$, we also prove an upper bound for the Hausdorff dimension of those vectors $e \in S^{2} \cap V$ with $\Hd \rho_{e}(K) \leq t < \Hd K$.
\end{abstract}

\maketitle

\tableofcontents

\section{Introduction}

The purpose of this paper is to investigate a connection between $1$-rectifiable families of projections onto lines in $\R^{3}$, and circular Kakeya problems in $\R^{2}$. The connection is not too complicated, at least on a heuristic level, but seems have gone unnoticed so far. Informally, we demonstrate that the two problems are of the same order of difficulty. The relevant circular Kakeya problem was solved by T. Wolff \cite{Wo2} in 1997. Building on his methods, we manage to gain new insight about projections.

We start by introducing the projection problem, in somewhat more generality than we will eventually need. Consider a $\calC^{2}$-curve $\gamma \colon J \to S^{2}$, where $J \subset \R$ is a bounded open interval, and $S^{2}$ is the unit sphere in $\R^{3}$. Following the framework introduced by K. F\"assler and the second author in \cite{FO}, we assume that $\gamma$ satisfies the following curvature condition:
\begin{equation}\label{curvature}
  \spa\{\gamma(\theta),\dot{\gamma}(\theta),\ddot{\gamma}(\theta)\} = \R^{3}, \qquad \theta \in J.
\end{equation}
A simple consequence of \eqref{curvature} is that $\gamma(I)$ cannot be contained in a fixed $2$-dimensional subspace for any interval $I \subset J$. The curve $\gamma$ gives rise to a $1$-rectifiable family of orthogonal projections $\rho_{\theta} \colon \R^{3} \to \R$ onto the $1$-dimensional subspaces spanned by $\gamma(\theta)$:
\begin{displaymath}
  \rho_{\theta}(z) := \gamma(\theta) \cdot z.
\end{displaymath}
It seems plausible to conjecture that a Marstrand-type projection theorem should hold for the mappings $\rho_{\theta}$. The following is the first part of \cite[Conjecture 1.6]{FO}, $\Hd$ denoting the Hausdorff dimension: 

\begin{conjecture}\label{mainC}
  Suppose that $\gamma$ is a $\mathcal{C}^2$-curve on $S^2$ satisfying \eqref{curvature} and $\rho_\theta \colon \R^3 \to \R$ is the family of orthogonal projections onto the $1$-dimensional subspaces spanned by $\gamma(\theta)$. If $K \subset \R^{3}$ is an analytic set, then $\Hd \rho_{\theta}(K) = \min\{\Hd K,1\}$ for almost every $\theta \in J$.
\end{conjecture}

The curvature condition \eqref{curvature} is necessary for any positive results. For instance, the curve $\gamma(\theta) = (\cos \theta, \sin \theta,0)$ evidently fails \eqref{curvature}, and every projection $\rho_{\theta}$ maps the set $\{(0,0,r) : r \in \R\}$ onto the singleton $\{0\}$. On the other hand, the prototypical example of a curve $\gamma$ satisfying \eqref{curvature} is given by
\begin{equation}\label{specGamma}
  \gamma(\theta) = \tfrac{1}{\sqrt{2}}(\cos \theta, \sin \theta, 1), \qquad \theta \in [0,2\pi).
\end{equation}
Note that the trace of $\gamma$ lies completely on the plane $\{(x_1,x_2,\tfrac{1}{\sqrt{2}}) : x_1,x_2\in\R\}$, but intersects every $2$-dimensional subspace at most twice. The existing results about, and around, Conjecture \ref{mainC} can be summarised as follows, $K$ denoting an analytic set in $\R^{3}$:
\begin{itemize}
  \item[(i)] It is easy to prove that if $\Hd K \le \tfrac12$, then $\Hd \rho_{\theta}(K) = \min\{\Hd K,\tfrac{1}{2}\}$ for almost every $\theta \in J$; see \cite[Proposition 1.5]{FO}. 
  \item[(ii)] If $\Hd K > \tfrac{1}{2}$, then the packing dimension of $\rho_{\theta}(K)$ strictly exceeds $\tfrac{1}{2}$ for almost every parameter $\theta \in J$; see \cite[Theorem 1.7]{FO}.
  \item[(iii)] The second author \cite[Theorem 1.9]{O} proved the Hausdorff dimension analogue of the previous result, but only for the special curve \eqref{specGamma}. 
  \item[(iv)] Both papers \cite{FO} and \cite{O}, and also the paper \cite{OO} by D. Oberlin and R. Oberlin, prove analogous results for projections onto the perpendicular planes $\spa \gamma(\theta)^{\perp}$.
  \item[(v)] Very recently, C. Chen \cite[Theorem 1.3]{Ch} showed that there exist $1$-dimensional (but not $1$-rectifiable) families of lines in $\R^{3}$, which satisfy Marstrand's projection theorem in the same sense as Conjecture \ref{mainC}. We will discuss C. Chen's result a bit further in Section \ref{s:tangencyParameter} below.
\end{itemize}
The reader is also referred to \cite[Section 5.4]{Mat2} for a related discussion. Our main result in the present paper solves Conjecture \ref{mainC} for the special curve \eqref{specGamma} studied in \cite{O}.

\begin{thm}\label{main}
 Suppose that $\gamma \colon [0,2\pi) \to S^2$ is the curve satisfying \eqref{specGamma}. If $K \subset \R^{3}$ is an analytic set, then $\Hd \rho_{\theta}(K) = \min\{\Hd K,1\}$ for almost every $\theta \in [0,2\pi)$. 
\end{thm}

In fact, we derive Theorem \ref{main} from the more precise result below:

\begin{thm}\label{main2}
  Suppose that $\gamma \colon [0,2\pi) \to S^2$ is the curve satisfying \eqref{specGamma}. If $K \subset \R^{3}$ is an analytic set with $0 < \Hd K \le 1$ and $0 \leq t < \Hd K$, then $\Hd \rho_{\theta}(K) \ge t$ for all $\theta \in [0,2\pi) \setminus E$, where
  \begin{displaymath}
      \Hd E \leq \frac{\Hd K + t}{2\Hd K} < 1.
  \end{displaymath}
\end{thm}

\begin{remark}\label{rem1} We note that Theorems \ref{main} and \ref{main2} remain true with the special curve $\gamma$ replaced by any non-degenerate circle of the form $\eta = V \cap S^{2}$, where $V \subset \R^{3}$ is a $2$-plane, which is not a subspace. This version of the results was mentioned in the abstract. Clearly, one may assume that $V$ is a parallel to the $xy$-plane, at height $h \in (-1,1) \setminus \{0\}$, and then $\eta$ can be parametrised by $\eta(\theta) = (c_{h}\cos \theta,c_{h}\sin \theta, h)$, with $c_{h} = \sqrt{1 - h^{2}}$. Consequently,
\begin{displaymath} \eta(\theta) \cdot (z_{1},z_{2},z_{3}) = (c_{h}z_{1} \cos \theta + c_{h}z_{2} \sin \theta + hz_{3}) = 2^{1/2} [\gamma(\theta) \cdot (c_{h}z_{1},c_{h}z_{2},hz_{3})]. \end{displaymath}
This shows the the projections of any set $K \subset \R^{3}$ to the lines spanned by $e \in V \cap S^{2}$ are (up to scaling by $2^{1/2}$) the same as the projections of the $K_{h} = \{(c_{h}z_{1},c_{h}z_{2},hz_{3}) : (z_{1},z_{2},z_{3}) \in K\}$ to the lines spanned by $e \in \gamma$. Now, it remains to note that $\Hd K_{h} = \Hd K$ for any $h \in (-1,1) \setminus \{0\}$, and apply Theorems \ref{main} and \ref{main2}. \end{remark}

\begin{remark} Theorem \ref{main} has been recently applied to solve a variant of the Kakeya problem in the first Heisenberg group $\mathbb{H}$, see \cite{Li}. The author of \cite{Li} proves that if $K \subset \mathbb{H}$ is an arbitrary set containing a horizontal line segment of every orientation, then the (Heisenberg) Hausdorff dimension of $K$ is no smaller than $3$. \end{remark}

Let us next examine how Conjecture \ref{mainC} is connected to curvilinear Kakeya problems in $\R^2$. The circular Kakeya problem asks how large is the Hausdorff dimension of a planar set $B$ which contains a circle of every radius. In 1994, L. Kolasa and T. Wolff \cite{KW} first proved that $\Hd B \geq 11/6$, and in 1997, T. Wolff \cite{Wo2} obtained the optimal result $\Hd B = 2$. The paper \cite{KW} also contains the $11/6$-result for sets containing "generalised circles of every radius" (we refer the reader to \cite{KW} for the precise definitions). The optimal result $\Hd B = 2$ in this setting was obtained quite recently by J. Zahl \cite{Za}.

A natural generalisation of the problem above is the following. A circle $S(x,r) \subset \R^{2}$ determines uniquely its own radius and centre, so the points in $\R^{2} \times \R_{+}$ are in one-to-one correspondence with planar circles. Thus, we can say that a family $\calS$ of planar circles is compact (or Borel, analytic, or $s$-dimensional), if the corresponding pairs $(x,r) \in \R^{2} \times \R_{+}$ form a compact (or respectively Borel, analytic, or $s$-dimensional) subset of $\R^{3}$. We denote the Hausdorff dimension of a circle family $\calS$ by $\Hd \calS := \Hd \{(x,r) \in \R^{3} : S(x,r) \in \calS\}$. Thus, by definition, each family $\calS$ of circles containing a circle of every radius evidently satisfies $\Hd \calS \geq 1$. 

Now, assume that $\calS$ is an analytic family of circles. What can be said about the dimension of $\cup \calS := \bigcup_{S \in \calS} S$? The answer does not appear to be stated explicitly in the literature, but the existing methods yield $\Hd \cup \calS = \min\{\Hd \calS + 1,2\}$ in this situation. We were informed by A. M\'ath\'e that this follows from a slight generalisation of Theorem 2.9 in T. Keleti's survey \cite{Ke}, combined with Corollary 3 in T. Wolff's deep paper \cite{Wo3}. Already in Wolff's earlier paper \cite[Appendix A]{Wo2}, he proved a slightly weaker variant: if the set of centres of the circles in $\mathcal{S}$ has dimension $\alpha \in (0,1]$, then $\Hd \cup \mathcal{S} \geq 1 + \alpha$.

As a corollary of the techniques in the present paper, we are able to give an elementary proof (avoiding the techniques of \cite{Wo3}) of the full result:

\begin{thm}\label{circleUnion}
  If $\calS$ is an analytic family of planar circles, then $\Hd \cup \calS = \min\{\Hd \calS + 1,2\}$.
\end{thm}

We should perhaps emphasise that even though Theorem \ref{circleUnion} can be deduced from \cite{Wo3} via \cite[Theorem 2.9]{Ke}, the same is not true of Theorem \ref{main}, as far as we know. 

Let us finally explain the connection to Conjecture \ref{mainC}. Let $\gamma \colon J \to S^{2}$ be a curve satisfying the non-degeneracy hypothesis \eqref{curvature}. For each $z \in \R^{3}$, consider the planar curve
\begin{displaymath}
  \Gamma(z) := \{(\theta,\rho_{\theta}(z)) : \theta \in J\}.
\end{displaymath}
For the special curve $\gamma(\theta) = \tfrac{1}{\sqrt{2}}(\cos \theta,\sin \theta,1)$ and $z = (x_1,x_2,r) \in \R^{3}$, the set $\Gamma(z)$ is the graph of the function $x_1\cos \theta + x_2\sin \theta + r$ defined on $[0,2\pi)$; we will often refer to these curves as "sine waves". As one shifts $z$ around in $\R^{3}$, the wave $\Gamma(z)$ changes. Note that the same is not true for the degenerate curve $\gamma(\theta) = (\cos \theta, \sin \theta,0)$, as $\Gamma(x_1,x_2,r)$ is then independent of $r$. 

The reader should now think that $\Gamma(z)$ is a "circle" parametrised by $z$. If $K \subset \R^{3}$ is an analytic set, then one might expect, based on Theorem \ref{circleUnion}, that the union $\bigcup_{z\in K}\Gamma(z)$ has Hausdorff dimension $\min\{\Hd K + 1,2\}$. The crucial observation here is that if $L_\theta = \{\theta\}\times\R \subset \R^2$ is the vertical line at $\theta$, then the vertical intersections
\begin{equation}\label{slicesEq}
  L_{\theta} \cap \bigcup_{z \in K} \Gamma(z) = \{(\theta,\rho_{\theta}(z)) : z \in K\}, \qquad \theta \in J,
\end{equation}
are isometric to the projections $\rho_{\theta}(K)$. Thus, if the union $\bigcup_{z \in K} \Gamma(z)$ is $s$-dimensional, with $s \geq 1$, then, by a Fubini-type argument, many projections $\rho_{\theta}(K)$ should have dimension $s - 1$. Strictly speaking this is not correct, since there is no such Fubini theorem for the Hausdorff dimension. Regardless, this gives a reasonable heuristic why Conjecture \ref{mainC} should hold for the projections $\rho_{\theta}$.

Our main result, Theorem \ref{main}, makes the above heuristic rigorous for the curve $\gamma(\theta) = \tfrac{1}{\sqrt{2}}(\cos \theta,\sin \theta,1)$. Observe also that, as an immediate corollary of Theorem \ref{main} and \eqref{slicesEq}, the union $\bigcup_{z \in K} \Gamma(z)$ has Hausdorff dimension $\min\{\Hd K + 1,2\}$; this corresponds to Theorem \ref{circleUnion} for the waves $\Gamma(z)$.

\subsection{Further directions}

It seems plausible that the strategy in this paper, combined with the "cinematic curvature" machinery developed by L. Kolasa and T. Wolff \cite{KW} and J. Zahl \cite{Za1,Za}, could be stretched to prove Conjecture \ref{mainC} for all curves satisfying \eqref{curvature}. There are several technical obstacles, however. One is quite simply verifying (rigorously) the "cinematic curvature hypothesis", see \cite[page 124]{KW}, for the relevant curves, and making sure that the tangency parameter "$\Delta$" in \cite{KW} coincides with the one we introduce in this paper. Another obstacle is verifying that \cite[Lemma 11]{Za} works under the assumption that the "generalised circles" in question are merely $\delta$-separated (and not necessarily $\delta$-separated in the radial variable); this would be needed for the generalised version of Lemma \ref{wolffMeasures} below. J. Zahl [personal communication] has informed us that the proof of \cite[Lemma 11]{Za} does not really rely on the radial separation, but verifying this carefully would result in a fairly long paper.

Another natural question arising from Theorem \ref{main} is the following: if $\Hd K > 1$, then is it true that $\mathcal{H}^{1}(\rho_{\theta}(K)) > 0$ for almost every $\theta \in [0,2\pi)$? This seems plausible, but does not follow from the method of this paper. Given the analogy with circle packing problems discussed above, this result would correspond to the fact that $\Hd \mathcal{S} > 1$ implies $\mathcal{L}^{2}(\cup \mathcal{S}) > 0$. This result established by Wolff \cite{Wo3} in 2000. It requires a combination of Fourier-analytic techniques with the incidence geometric ideas behind Theorem \ref{circleUnion}.

\subsection{Notation}\label{s:notation}

We generally denote points of $\R^{3}$ by $z,z'$, and points in $\R^{2}$ by $x,y$. A closed ball of radius $r > 0$ and centre $z \in \R^{d}$ is denoted by $B(z,r)$. A planar circle of radius $r > 0$ and centre $x \in \R^{2}$ is denoted by $S(x,r)$.

For $A,B > 0$, we use the notation $A \lesssim_{p} B$ to signify that there exists a constant $C \geq 1$, depending only on the parameter $p$, such that $A \leq CB$. If no "$p$" is specified, then the constant $C$ is absolute. We abbreviate the two-sided inequality $A \lesssim_{p} B \lesssim_{q} A$ by $A \sim_{p,q} B$. In general, the letter "$C$" stands for a large constant, whose value may change from line to line inside the proofs. More essential constants will be indexed $C_{1},C_{2},\ldots$ In addition to the "$\lesssim$" notation, we will also need the "$\lessapprox$" notation: this notation is always associated with a "scale" parameter $\delta \in (0,1]$, which will be clear from context. Given this parameter $\delta$, the notation $A \lessapprox B$ means that there exists an absolute constant $C \geq 1$ such that $A \leq C(\log(1/\delta))^{C} B$. In this paper, "$\log$" refers to logarithm of base $2$. The two-sided inequality $A \lessapprox B \lessapprox A$ is abbreviated to $A \approx B$.

The notation $\calH^{s}$ stands for the $s$-dimensional Hausdorff measure, and $\mathcal{H}^{s}_{\infty}$ stands for $s$-dimensional Hausdorff content. The notation $|\cdot |$ can refer to the norm of a vector, or the Lebesgue measure, or the counting measure, depending on the context. 

\subsection*{Acknowledgement} We are grateful to anonymous referees for reading the paper very carefully, and for providing a large number of helpful comments and small corrections. 

\section{The tangency parameter}\label{s:tangencyParameter}

A great deal of what follows has nothing to do with the curve $\gamma(t) = \tfrac{1}{\sqrt{2}}(\cos t,\sin t,1)$, and would work equally well under the general curvature hypothesis \eqref{curvature}. For the moment, we fix any $\calC^{2}$-curve $\gamma \colon J \to S^{2}$ satisfying the curvature condition \eqref{curvature} on $J$. For convenience, we also assume that $\gamma$, $\dot{\gamma}$, and $\ddot{\gamma}$ extend continuously to the closure $\overline{J}$, and \eqref{curvature} holds on $\overline{J}$. 

To motivate the following definitions, we recall a part of Marstrand's classical projection theorem in $\R^{3}$; see \cite{Mar}. Let $e \in S^{2}$, and let $\pi_{e} \colon \R^{3} \to \R$ be the orthogonal projection onto the line spanned by $e$, that is, $\pi_{e}(x) = x \cdot e$. If $K \subset \R^{3}$ is analytic, then Marstrand's classical projection theorem guarantees that $\mathcal{H}^2|_{S^2}$ almost every projection $\pi_{e}(K)$ satisfies $\Hd \pi_{e}(K) = \min\{\Hd K,1\}$. A fundamental ingredient in the proof of this result is the following estimate:
\begin{equation}\label{sublevel}
  \mathcal{H}^2(\{e \in S^{2} : |\pi_{e}(z)| \leq \delta\}) \lesssim \delta/|z|, \qquad z \in \R^{3} \setminus \{0\}.
\end{equation}
In fact, whenever \eqref{sublevel} holds for a (non-trivial) measure $\sigma$ on $S^{2}$, then the usual proof of Marstrand's theorem works for this measure $\sigma$. In \cite{Ch}, C. Chen found that there are $\alpha$-Ahlfors-David regular measures $\sigma$ on $S^{2}$ with $\alpha$ arbitrarily close to $1$, which satisfy \eqref{sublevel}. 

The main difficulty in dealing with the projections $\rho_{\theta}(z) = \gamma(\theta) \cdot z$, $\theta \in J$, is that non-trivial measures on the curve $\gamma \subset S^{2}$ fail to satisfy \eqref{sublevel} (here we also use "$\gamma$" to denote the trace of $\gamma$). In fact, the length measure $\sigma = \calH^{1}|_{\gamma}$ only satisfies the uniform bound \eqref{sublevel} with the right hand side replaced by $(\delta/|z|)^{1/2}$; see \cite[proof of Lemma 3.1]{FO}. As a corollary, the projections $\rho_{\theta}$ conserve almost surely the dimension of at most $\tfrac{1}{2}$-dimensional analytic sets; see \cite[Proposition 1.5]{FO}.

The above explanation implies that, if one wants to consider sets of dimension higher than $\tfrac{1}{2}$, more careful analysis is required. Heuristically, the main observation here is that even though the best possible \textbf{uniform} estimate in \eqref{sublevel} is too weak for our purposes, a much stronger bound holds for "most" points $z \in \R^{3}$. For example, consider the projections $\rho_{\theta}$ associated with the special curve $\gamma(\theta) = \tfrac{1}{\sqrt{2}}(\cos \theta, \sin \theta, 1)$. If $z = (0,0,r)$ with $|r| \sim 1$, then $|\rho_{\theta}(z)| = \tfrac{|r|}{\sqrt{2}} \gtrsim 1$ for all $\theta \in [0,2\pi)$. In particular, the dangerous set on the left hand side of \eqref{sublevel} is empty altogether for $\delta > 0$ sufficiently small.

For each $z \in \R^{3} \setminus \{0\}$, the decay of $\calH^{1}(\{\theta \in J : |\rho_{\theta}(z)| \leq \delta\})$ depends on the maximum order of zeros of the real function
\begin{displaymath}
  \theta \mapsto \rho_{\theta}(z).
\end{displaymath}
As we just saw, the function need not have any zeros, but it can easily have zeros of either first or second order. Third order zeros are ruled out by the curvature condition \eqref{curvature}. If the zeros had order at most one, then \eqref{sublevel} would hold, and hence the second order zeros are revealed as the main adversary. So, when do second order zeros occur? Recall that $\rho_{\theta}(z) = \gamma(\theta) \cdot z$. Hence, $\rho_{\theta}(z) = 0 = \partial_{\theta} \rho_{\theta}(z)$, if and only if $z \perp \gamma(\theta)$ and $z \perp \dot{\gamma}(\theta)$. This is further equivalent to 
\begin{displaymath}
  \pi_{V_{\theta}}(z) = 0,
\end{displaymath}
where $V_{\theta} = \spa\{\gamma(\theta),\dot{\gamma}(\theta)\}$ and $\pi_{V_\theta}$ is the orthogonal projection onto the plane $V_\theta$. So, the function $\theta \mapsto \rho_{\theta}(z)$ has a second order zero at some $\theta \in \overline{J}$, if and only if
\begin{equation}\label{form77}
  \Delta(z) := \min_{\theta \in \overline{J}} |\pi_{V_{\theta}}(z)| = 0.
\end{equation}
The quantity $\Delta(z)$ is the \emph{tangency parameter} of $\gamma$ at $z$. In practice, "almost" second order zeros are also a challenge in the proofs below. It turns out that the size of $\Delta(z)$ is a good tool for quantifying the word "almost".

\subsection{Geometric interpretation of the tangency parameter}\label{geomInterpretation}

Condition \eqref{form77} tells us when second order zeros occur, but we will now give a more geometric characterisation. We only consider the special curve $\gamma(\theta) = \tfrac{1}{\sqrt{2}}(\cos \theta,\sin \theta,1)$. By a straightforward calculation, we see that $\dot{\gamma}(\theta) = \tfrac{1}{\sqrt{2}}(-\sin \theta, \cos \theta,0)$ and 
\begin{displaymath}
  \eta(\theta) := \gamma(\theta) \times \dot{\gamma}(\theta) = -\tfrac{1}{2}(\cos \theta,\sin \theta,-1), \quad \theta \in [0,2\pi).
\end{displaymath}
Thus, $\pi_{V_{\theta}}(z) = 0$, if and only if $z$ is parallel to $\ell_{\theta} := \spa\{\eta(\theta)\} = V_{\theta}^{\perp}$, and hence $\Delta(z) = 0$, if and only if\footnote{It may seem like a natural question, whether we could now prove Theorem \ref{main} separately for sets lying on $\calC$, and sets avoiding $\calC$. Unfortunately, the classical proof of Marstrand's theorem requires \eqref{sublevel} to hold for all $z = z_{1} - z_{2} \in (K - K) \setminus \{0\}$, and not just $z \in K \setminus \{0\}$. So, the classical proof would work for such sets $K$, where every non-zero vector in $K - K$ forms an angle $\epsilon > 0$ with the conical surface $\calC$. It would be interesting to understand the structure of such sets.}
\begin{equation}\label{form36a}
  z \in \calC := \bigcup_{\theta \in [0,2\pi)} \ell_{\theta} = \{(x,r) \in \R^{3} : |x| = |r|\}.
\end{equation}
Here $\Delta(z)$ is defined as in \eqref{form77}, with $J = [0,2\pi)$. There is another interesting (and useful) interpretation for $\Delta(z)$. Pick $\theta \in [0,2\pi)$ such that 
\begin{displaymath}
  \dist(z,\ell_{\theta}) = |\pi_{V_{\theta}}(z)| = \Delta(z).
\end{displaymath}
Then, pick $z' = (y,s) \in \ell_{\theta}$ with $|z - z'| = \Delta(z)$, and note that $|y| = |s|$ by \eqref{form36a}. Write $z = (x,r)$. Since $|x - y| \leq \Delta(z)$ and $|r - s| \leq \Delta(z)$, we infer that
\begin{equation}\label{form40a}
  \Delta'(z) := ||x| - |r|| \leq |x - y| + |r - s| \leq 2\Delta(z).
\end{equation} 
We also note that a converse to \eqref{form40a} holds. Fix $z = (x,r) \in \R^{3}$, and let $x = (r' \cos \theta,r' \sin \theta)$ in polar coordinates with $\theta \in [0,2\pi)$ and $r' = |x| \geq 0$. We note that $(|r| \cos \theta, |r| \sin \theta, r) \in \calC$, and
\begin{displaymath}
  |z - (|r| \cos \theta, |r| \sin \theta, r)| = |r' - |r|| = ||x| - |r|| = \Delta'(z).
\end{displaymath} 
This means that $z$ is at distance $\Delta'(z)$ from one of the lines $\ell_{\theta} = V_{\theta}^{\perp}$ on $\calC$, and hence
\begin{equation}\label{deltaPrimeA}
  \Delta(z) \leq \Delta'(z).
\end{equation}
Consequently, by \eqref{form40a} and \eqref{deltaPrimeA}, the numbers $\Delta(z)$ and $\Delta'(z)$ are comparable, and $\Delta(z) = 0$, if and only if $\Delta'(z) = 0$. This is useful, because the number $\Delta'(z)$ plays a major role in Wolff's investigation of circular Kakeya problems; see for instance \cite[Lemma 3.1]{Wo}. If $z_{1} = (x_{1},r_{1}),z_{2} = (x_{2},r_{2}) \in \R^{3}$ are distinct points with $r_{1},r_{2} \geq 0$, then 
\begin{displaymath}
  0 = \Delta'(z_{1} - z_{2}) = ||x_{1} - x_{2}| - |r_{1} - r_{2}||,
\end{displaymath}
if and only if the planar circles $S(x_{1},r_{1})$ and $S(x_{2},r_{2})$ are internally tangent.

\section{Geometric lemmas}

For technical reasons to be clarified in this section, it is easier (and sufficient) to prove Theorem \ref{main} for every sufficiently short compact subinterval $J \subset [0,2\pi)$ separately. We will adopt the notation
\begin{equation}\label{tangency}
  \Delta_{J}(z) = \min_{\theta \in J} |\pi_{V_{\theta}}(z)|, 
\end{equation}
where, as before, $V_{\theta} = \spa\{\gamma(\theta),\dot{\gamma}(\theta)\}$. Since $\{\gamma(\theta),\dot{\gamma}(\theta)\}$ is an orthonormal basis of $V_\theta$ (we can achieve this by re-parametrising $\gamma$ by arc-length), we have the estimate
\begin{equation} \label{Delta-gamma-est}
  \Delta_{J}(z) \le |\pi_{V_\theta}(z)| \le |\gamma(\theta) \cdot z| + |\dot{\gamma}(\theta) \cdot z| \le 2|\pi_{V_\theta}(z)|
\end{equation}
for all $\theta \in J$. We also trivially have
\begin{displaymath}
  \Delta_{J}(z) \leq |z|.
\end{displaymath}
The definition \eqref{tangency} makes sense for the general $\gamma$ satisfying the curvature condition \eqref{curvature}, as long as $J$ is contained in the domain of definition. In fact, until further notice, we work in that generality: the only standing assumptions are that $\gamma$, $\dot{\gamma}$, and $\ddot{\gamma}$ are continuous and well-defined on a compact interval $J$, and the curvature condition \eqref{curvature} is satisfied on $J$.

The compactness of $J$ and the curvature condition \eqref{curvature} together imply that there exists a constant $\kappa = \kappa(\gamma,J) > 0$ such that
\begin{equation}\label{kappa}
  \max\{|\gamma(\theta) \cdot w|,|\dot{\gamma}(\theta) \cdot w|, |\ddot{\gamma}(\theta) \cdot w|\} \geq \kappa, \qquad (w,\theta) \in S^{2} \times J.
\end{equation}
The following lemma is a simple consequence of uniform continuity:

\begin{lemma}\label{uCont}
  There exists a constant $\lambda = \lambda(\kappa,\gamma) > 0$ with the following property: If $I \subset J$ is an interval of length $|I| \leq \lambda$, $z \in \R^{3}$, and $\theta \mapsto \phi_{z}(\theta)$ is one of the functions $\phi_{z}(\theta) = \gamma(\theta) \cdot z$ or $\phi_{z}(\theta) = \dot{\gamma}(\theta) \cdot z$ or $\phi_{z}(\theta) = \ddot{\gamma}(\theta) \cdot z$, then one of the following alternatives holds (depending on the choice of $I$ and $\phi_{z}$):
  \begin{itemize}
    \item[\textup{(S)}] $|\phi_{z}(\theta)| < \kappa |z|$ for all $\theta \in I$.
    \item[\textup{(L)}] $|\phi_{z}(\theta)| \geq \kappa |z|/2$ for all $\theta \in I$.
  \end{itemize} 
\end{lemma}

\begin{proof}
  The maps $(w,\theta) \mapsto \gamma(\theta) \cdot w$, $(w,\theta) \mapsto \dot{\gamma}(\theta) \cdot w$ and $(w,\theta) \mapsto \ddot{\gamma}(\theta) \cdot w$ are uniformly continuous on $S^{2} \times J$. So, there is a constant $\lambda$ such that if $|(w,\theta) - (w',\theta')| \leq \lambda$, then $|\gamma(\theta) \cdot w - \gamma(\theta') \cdot w'| \leq \kappa/2$, and the same holds with $\gamma$ replaced by either $\dot{\gamma}$ or $\ddot{\gamma}$. Now, fix $I \subset J$ with $|I| \leq \lambda$, $z \in \R^{3}$, and $\phi_{z}$. Assume, for instance, that $\phi_{z}(\theta) = \gamma(\theta) \cdot z$. If $z = 0$, then evidently the alternative (L) holds. Otherwise, assume that $|z| > 0$, and alternative (L) fails. So, there exists $\theta_{0} \in I$ such that $|\phi_{z}(\theta_{0})| < \kappa|z|/2$. Then, if $\theta \in I$ is arbitrary, we have $|((z/|z|),\theta) - ((z/|z|),\theta_{0})| \leq \lambda$, and so
  \begin{displaymath}
    \frac{|\phi_{z}(\theta)|}{|z|} = \left|\gamma(\theta) \cdot \frac{z}{|z|} \right| \leq \left|\gamma(\theta_{0}) \cdot \frac{z}{|z|} \right| + \left| \gamma(\theta) \cdot \frac{z}{|z|} - \gamma(\theta_{0}) \cdot \frac{z}{|z|} \right| < \frac{\kappa}{2} + \frac{\kappa}{2} = \kappa.
  \end{displaymath}
  This means that alternative (S) holds for $I$ and $\phi_{z}$.
\end{proof}

Combined with \eqref{kappa}, the previous lemma has the following useful consequence:

\begin{lemma}
  Let $\lambda > 0$ be as in Lemma \ref{uCont}. If $I \subset J$ is an interval of length $|I| \leq \lambda$ and $z \in \R^{3} \setminus \{0\}$, then the map $\theta \mapsto \gamma(\theta) \cdot z$ has at most two zeros on $I$. Moreover, if $\theta \mapsto \dot{\gamma}(\theta) \cdot z$ has two zeros on $I$, then the alternative \textup{(L)} holds for $I$ and $\theta \mapsto \gamma(\theta) \cdot z$.
\end{lemma}

\begin{proof}
  We start with the second claim. Assume that $|I| \leq \lambda$ and $\theta \mapsto \dot{\gamma}(\theta) \cdot z$ has two zeros on $I$, for some $z \in \R^{3} \setminus \{0\}$. This implies, by Rolle's theorem, that $\theta \mapsto \ddot{\gamma}(\theta) \cdot z$ has a zero on $I$. Now Lemma \ref{uCont} implies that the alternative (S) holds for $I$ and both $\theta \mapsto \dot{\gamma}(\theta) \cdot z$ and $\theta \mapsto \ddot{\gamma}(\theta) \cdot z$. Consequently, by \eqref{kappa}, we have $|\gamma(\theta) \cdot z| \geq \kappa |z|$ for all $\theta \in I$, so alternative (L) holds for $\theta \mapsto \gamma(\theta) \cdot z$. 

  The first claim follows from the second one: If $\theta \mapsto \gamma(\theta) \cdot z$ had three zeros on $I$, then $\theta \mapsto \dot{\gamma}(\theta) \cdot z$ would have two zeros on $I$ again by Rolle's theorem. But then, by the second claim, $\theta \mapsto \gamma(\theta) \cdot z$ satisfies the alternative (L) on $I$, and hence cannot have zeros on $I$.
\end{proof}

Since the short subintervals $I \subset J$ have such pleasant properties, we restrict our attention to one of them. For notational convenience, we redefine $J$ to be any subinterval of the initial interval of length $|J| \leq \lambda/2$, and such that $2J$ is still contained inside the initial interval. This change in notation also affects the definition of $\Delta_{J}$ in \eqref{tangency}.

\begin{assumption}\label{ass1}
  We assume that the interval $2J$ satisfies the conclusion of Lemma \ref{uCont}: for every $z \in \R^{3}$, and each of the three possible choices of $\phi_{z}$, either alternative (L) or (S) is satisfied on the interval $2J$.
\end{assumption}

The next lemma is a close relative of Lemma 3.1 in \cite{KW}, and proof is virtually the same.

\begin{lemma}\label{Edelta}
  Fix $\delta > 0$ and $z \in \R^{3}$ with $|z| \geq C\delta$, where $C = C(\gamma,J) \geq 1$ is a sufficiently large constant. Define $E_{\delta}(z) := \{\theta \in J/2 : |\gamma(\theta) \cdot z| \le \delta\}$. 
  \begin{itemize}
    \item[(1)] The set $E_{\delta}(z)$ is contained in a single interval of length at most a constant times 
    \begin{displaymath} \sqrt{(\Delta_{J}(z) + \delta)/|z|}. \end{displaymath}
    Moreover, if $\Delta_{J}(z) \leq |z|/C$, and $C = C(\gamma,J)$ is sufficiently large, then this interval can be centred at a point $\theta_{0} \in 2J$ with $\dot{\gamma}(\theta_{0}) \cdot z = 0$ and $|\pi_{V_{\theta_{0}}}(z)| \lesssim \Delta_{J}(z)$.
    \item[(2)] The set $E_{\delta}(z)$ consists of at most two intervals $I_{1},I_{2}$, whose lengths are bounded by
    \begin{equation*}
      |I_{j}| \lesssim \frac{\delta}{\sqrt{(\Delta_{J}(z) + \delta)|z|}}.
    \end{equation*}
  \end{itemize}
  The implicit constants in the estimates above depend only on $\gamma$ and $J$.
\end{lemma}

\begin{proof}
  Write $\Delta := \Delta_{J}(z)$. First of all, we may assume that
  \begin{equation}\label{form93}
    \Delta \leq c|z|
  \end{equation} 
  for a suitable small constant $c = 1/C(\gamma,J) \in (0,\kappa/4)$, to be determined a bit later. Indeed, otherwise $|\gamma(\theta) \cdot z| + |\dot{\gamma}(\theta) \cdot z| \ge \Delta > c|z| \geq 2\delta$ for all $\theta \in J$ by \eqref{Delta-gamma-est} and the assumption $|z| \geq C\delta$, and in particular $|\dot{\gamma}(\theta) \cdot z| \gtrsim_{\gamma,J} |z|$ for $\theta \in E_{\delta}(z)$. If this is the case, both claims of the lemma are easy to verify.

  Since $c < \kappa/4$, the estimates \eqref{Delta-gamma-est} and \eqref{form93} imply that $|\gamma(\theta) \cdot z| + |\dot{\gamma}(\theta) \cdot z| \le 2\Delta < \kappa|z|/2$ for some $\theta$. Therefore, both $\theta \mapsto \gamma(\theta) \cdot z$ and $\theta \mapsto \dot{\gamma}(\theta) \cdot z$ satisfy the alternative (S) on $2J$. Hence, by the quantitative curvature condition \eqref{kappa}, we have
  \begin{equation}\label{form26}
    |\ddot{\gamma}(\theta) \cdot z| \geq \kappa|z|, \qquad \theta \in 2J.
  \end{equation}
  Thus, $\theta \mapsto \gamma(\theta) \cdot z$ is either strictly convex or strictly concave on $2J$, and $E_{\delta}(z)$ consists of at most two intervals $I_{1}$ and $I_{2}$. Thus, the situation is reduced to the fairly simple case depicted in Figure \ref{fig2}.

  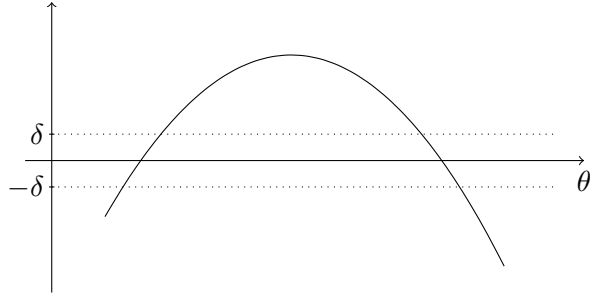
\begin{figure}[t!]
  \begin{center}
    \begin{tikzpicture}[scale=0.7]
      \draw [->] (-0.5,0) -- (10,0) node [below] {$\theta$};
      \draw [->] (0,-2.5) -- (0,3) node [] {};
      \draw (-0.05,-0.5) -- (0.05,-0.5) node [left] {$-\delta$};
      \draw (-0.05,0.5) -- (0.05,0.5) node [left] {$\delta$};
      \draw [domain=1:8.5,smooth,variable=\x] plot ({\x},{(-0.25)*(\x-4.5)^2+2});
      \draw [dotted] (0,-0.5) -- (9.5,-0.5);
      \draw [dotted] (0,0.5) -- (9.5,0.5);
    \end{tikzpicture}
    \caption{The picture depicts the map $\theta \mapsto \gamma(\theta) \cdot z$ and the set $E_\delta(z)$ in Lemma \ref{Edelta}.}
    \label{fig2}
  \end{center}
  \end{figure}

  Let $\theta_{\Delta} \in J$ be such that
  \begin{displaymath}
    |\pi_{V_{\theta_{\Delta}}}(z)| = \Delta.
  \end{displaymath}
  Then \eqref{Delta-gamma-est} implies that $|\dot\gamma(\theta_\Delta) \cdot z| \le 2\Delta$. By \eqref{form26}, and assuming that $c$ in \eqref{form93} satisfies $c < \kappa|J|/10$, the mapping $\theta \mapsto \dot{\gamma}(\theta) \cdot z$ has a unique zero at some point $\theta_{0} \in 2J$ with $|\theta_{0} - \theta_{\Delta}| \leq 2\Delta/(\kappa |z|) < |J|/5$. Observe that
  \begin{equation}\label{form27}
    |\gamma(\theta_{0}) \cdot z| \leq |\gamma(\theta_{\Delta}) \cdot z| + \int_{\theta_{\Delta}}^{\theta_{0}} |\dot{\gamma}(s) \cdot z| \, \dd s \leq \Delta + |z||\theta_{0} - \theta_{\Delta}| \leq C\Delta,
  \end{equation}
  where $C = C_{\gamma,J} \geq 1$, so in particular $|\pi_{V_{\theta_{0}}}(z)| \lesssim \Delta$.

  Write $[a,b] := 2J$. Note that neither $\theta \mapsto \dot{\gamma}(\theta) \cdot z$ nor $\theta \mapsto \ddot{\gamma}(\theta) \cdot z$ changes sign on $[a,\theta_{0}]$ or $[\theta_{0},b]$. Thus, if $\theta_{1} \in [\theta_{0},b]$, then we can use \eqref{form27} and \eqref{form26} to estimate
  \begin{equation*}
  \begin{split}
    |\gamma(\theta_{1}) \cdot z| &\geq \left| \int_{\theta_{0}}^{\theta_{1}} \dot{\gamma}(s) \cdot z \dd s \right| - C\Delta = \int_{\theta_{0}}^{\theta_{1}} |\dot{\gamma}(s) \cdot z| \dd s - C\Delta \\
    &= \int_{\theta_{0}}^{\theta_{1}} \int_{\theta_0}^{s} |\ddot{\gamma}(r) \cdot z| \dd r \dd s - C\Delta \ge \frac{\kappa |z|}{2}(\theta_{1} - \theta_{0})^{2} - C\Delta.
  \end{split}
  \end{equation*}
  Thus, $\theta_{1} \in E_{\delta}(z)$ can only occur, if $\kappa |z|(\theta_{1} - \theta_{0})^{2}/2 - C\Delta \leq \delta \leq C\delta$, which gives 
  \begin{displaymath}
    \theta_{1} - \theta_{0} \leq \left(\frac{2C}{\kappa} \right)^{1/2}\sqrt{(\Delta + \delta)/|z|}.
  \end{displaymath} 
  If $\theta_{1} \in [a,\theta_{0}]$, then a similar estimate holds for $\theta_{0} - \theta_{1}$. Hence
  \begin{equation}\label{form29}
    E_{\delta}(z) \subset B(\theta_{0},C\sqrt{(\Delta + \delta)/|z|})
  \end{equation}
  for $C = C_{\gamma,J} \geq 1$, as claimed (here, and in the remainder of the proof, the numerical value of "$C$" is allowed to change from line to line, but it will only depend on $\gamma,J$).

  To prove the second claim, recall that $E_{\delta}(z)$ consists of at most two intervals $I_{1}$ and $I_{2}$, which, by \eqref{form29}, are both located inside $B(\theta_{0},C\sqrt{(\Delta + \delta)/|z|})$. If $\Delta \leq 2\delta$, then the estimate $|I_{j}| \leq |B(\theta_{0},C\sqrt{(\Delta + \delta)/|z|})| \lesssim \sqrt{\delta}/\sqrt{|z|}$ gives the desired bound. So, we may assume that 
  \begin{equation}\label{form79}
    2\delta \leq \Delta \leq c|z|.
  \end{equation}
  Then, if $c > 0$ was taken small enough, depending on $J$, the diameter of the single interval in \eqref{form29} containing both $\theta_{0}$ and $E_{\delta}(z)$ is smaller than $|J|/10$. In particular, if $\theta_{0} \in 2J \setminus J$, then $E_{\delta}(z) \subset J/2$ is empty. So, we may assume that $\theta_{0} \in J$, which gives $|\gamma(\theta_{0}) \cdot z| \geq \Delta$ by \eqref{Delta-gamma-est} and recalling that $\dot{\gamma}(\theta_{0}) \cdot z = 0$. Observe that if $\theta_{1} \in E_{\delta}(z)$, then, by \eqref{form26},
  \begin{align*}
    \delta &\geq |\gamma(\theta_{1}) \cdot z| \geq \Delta - \int_{\theta_{0}}^{\theta_{1}} |\dot{\gamma}(s) \cdot z| \dd s \\
    &\geq \Delta - \int_{\theta_{0}}^{\theta_{1}} \int_{\theta_{0}}^{s} |\ddot{\gamma}(r) \cdot z| \dd r \dd s \geq \Delta - C|z|(\theta_{1} - \theta_{0})^{2}
  \end{align*}
  with $C = \|\ddot{\gamma}\|_{L^{\infty}(J)}$. By \eqref{form79}, this implies
  \begin{displaymath}
    |\theta_{1} - \theta_{0}| \gtrsim \sqrt{\frac{\Delta - \delta}{|z|}} \gtrsim \sqrt{\frac{\Delta}{|z|}}, \qquad \theta_{1} \in E_{\delta}(z).
  \end{displaymath}
  Using \eqref{form26}, we finally infer that
  \begin{equation*}
    |\dot{\gamma}(\theta_{1}) \cdot z| = \int_{\theta_0}^{\theta_1} |\ddot{\gamma}(s) \cdot z| \dd s \ge \kappa|z||\theta_1-\theta_0| \gtrsim \sqrt{\Delta |z|}
  \end{equation*}
  for $\theta_{1} \in E_{\delta}(z)$, which shows that $|I_j| \lesssim \delta/\sqrt{\Delta |z|}$. The proof is complete.
\end{proof}

\subsection{Tangency of circles}
In this section, we gather some estimates on the size and shape of intersections of (circular) annuli. These are harvested verbatim from T. Wolff's paper \cite{Wo3} and survey \cite{Wo}. 

\begin{definition}[The region $\mathbf{B}_{0}$]\label{B0}
  We write $\mathbf{B}_{0} \subset \R^{3}$ for the set 
  \begin{displaymath} \mathbf{B}_{0} = \{(x,r) \in \R^{3} : x \in B(0,\tfrac{1}{4}) \text{ and } \tfrac{1}{2} \leq r \leq 2\}.
  \end{displaymath}
\end{definition}

The set plays the role of "the unit ball" or "the unit cube" in the arguments below: geometric constants stay under control, as long as points are chosen from $\mathbf{B}_{0}$. The next result is from \cite{Wo}, and it is an analogue of Lemma \ref{Edelta} for circles (also the proof is fairly similar). To be precise, the statement of Lemma \ref{circleLemma} contains some details which are not explicit in the statement of \cite[Lemma 3.1]{Wo}, but are apparent from the proof.

\begin{lemma}[\mbox{\cite[Lemma 3.1]{Wo}}]\label{circleLemma}
  Assume that $S(x_{1},r_{1})$ and $S(x_{2},r_{2})$ are planar circles with $(x_{1},r_{1}),(x_{2},r_{2}) \in \mathbf{B}_{0}$. Let $\delta > 0$ and denote by $S^{\delta}(x,r)$ the $\delta$-annulus around the circle $S(x,r)$. Define\footnote{Recall Section \ref{geomInterpretation}.} $\Delta' := \Delta'((x_{1},r_{1})-(x_{2},r_{2})) = ||x_{1} - x_{2}| - |r_{1} - r_{2}||$ and write $t := |(x_{1},r_{1}) - (x_{2},r_{2})|$. Then 
  \begin{itemize}
    \item[(1)] $S^{\delta}(x_{1},r_{1}) \cap S^{\delta}(x_{2},r_{2})$ is contained in a ball centred at
    \begin{displaymath}
      \zeta(x_{1},x_{2}) := x_{1} + \sgn(r_{1} - r_{2})r_{1}\frac{x_{2} - x_{1}}{|x_{2} - x_{1}|},
    \end{displaymath}
    with radius at most a constant times $\sqrt{(\Delta' + \delta)/(t + \delta)}$. 
    \item[(2)] $S^{\delta}(x_{1},r_{1}) \cap S^{\delta}(x_{2},r_{2})$ is contained in the union of the $\delta$-neighbourhoods of at most two arcs on $S(x_{2},r_{2})$, both of length at most a constant times $\delta/\sqrt{(\Delta' + \delta)(t + \delta)}$. In particular, 
    \begin{displaymath}
      |S^{\delta}(x_{1},r_{1}) \cap S^{\delta}(x_{2},r_{2})| \lesssim \frac{\delta^{2}}{\sqrt{(\Delta' + \delta)(t + \delta)}}.
    \end{displaymath}
  \end{itemize}
\end{lemma}

An important special case of the lemma is when $\Delta' \leq \delta$: following T. Wolff \cite{Wo3}, we say that the two circles $S(x_{1},r_{1})$ and $S(x_{2},r_{2})$ are then $\delta$-\emph{incident}, and it follows from Lemma \ref{circleLemma}(1) that $S^{C\delta}(x_{1},r_{1}) \cap S^{C\delta}(x_{2},r_{2})$ can be covered by a single $\delta$-neighbourhood of a circular arc of length $\lesssim_{C} \sqrt{\delta/(t + \delta)}$. This numerology motivates the following definition (which is from \cite[Section 1]{Wo3}):

\begin{definition}[$(\delta,t)$-rectangles]
  Let $0 < \delta \leq t \leq 1$. A \emph{$(\delta,t)$-rectangle} $R \subset \R^{2}$ is a $\delta$-neighbourhood of a circular arc of length $\sqrt{\delta/t}$. Two $(\delta,t)$-rectangles are \emph{$C$-comparable}, if there is a single $(C\delta,t)$-rectangle containing both of them. Otherwise $R_{1}$ and $R_{2}$ are $C$-\emph{incomparable}. A circle $S(x,r)$ is \emph{$C$-tangent} to a $(\delta,t)$-rectangle, if $S^{C\delta}(x,r)$ contains $R$. Finally, fixing some large absolute constant $C_{0} \geq 1$, we say that two rectangles $R_{1},R_{2}$ are simply \emph{comparable}, if they are $C_{0}$-comparable. Similarly, a circle being \emph{tangent} to a rectangle refers to $C_{0}$-tangency. 
\end{definition}

We record a part of \cite[Lemma 1.5]{Wo3}:

\begin{lemma}[Incidence vs.\ tangency]\label{incidenceTangency}
  Assume that $S_{1} = S(x_{1},r_{2})$ and $S_{2} = S(x_{2},r_{2})$ satisfy the hypotheses of Lemma \ref{circleLemma}, with constants $t$ and $\Delta' \leq \delta$, so that the two circles are $\delta$-incident. Then, there exists a $(\delta,t)$-rectangle $R$ such that both $S_{1}$ and $S_{2}$ are tangent to $R$ (assuming that the constant $C_{0} \geq 1$ in the definition above was chosen large enough).
\end{lemma}

\subsection{Tangency of sine waves}\label{tangencySection}

In this section, we apply the discussion above to the special curve we are considering in the present paper, namely 
\begin{displaymath}
  \gamma(\theta) = \tfrac{1}{\sqrt{2}}(\cos \theta, \sin \theta, 1).
\end{displaymath}
We keep assuming that $J \subset 2J \subset [0,2\pi)$ is a compact subinterval such that Assumption \ref{ass1} holds for $2J$. Note that
\begin{equation}\label{form80}
  \Delta'(z) \leq 2\Delta(z) \leq 2\Delta_{J}(z)
\end{equation}
by \eqref{form40a}. The converse inequality $\Delta_{J}(z) \leq \Delta'(z)$ is no longer true. Heuristically, $\Delta_{J}(z_{1} - z_{2})$ only measures the tangency between certain arcs of $S(x_{1},r_{1}),S(x_{2},r_{2})$, determined by $J$; even if the circles $S(x_{1},r_{1}),S(x_{2},r_{2})$ happened to be tangent, that is $\Delta'(z_{1} - z_{2}) = 0$, the point of tangency need not occur on this arc.
We define
\begin{displaymath}
  \Gamma(z) := \Gamma_{J}(z) := \{(\theta,\rho_{\theta}(z)) : \theta \in \tfrac{J}{2}\} \quad \text{and} \quad \Gamma^{\delta}(z) := \{(\theta,\theta') \in \tfrac{J}{2} \times \R : |\rho_{\theta}(z) - \theta'| \leq \delta\}.
\end{displaymath}
So, formally, $\Gamma(z) = \Gamma^{0}(z)$.

For later application, we are interested in the following problem. Fix $\epsilon,t \in (0,1]$ with $2C\epsilon \leq t$, where $C = C(\gamma,J) \geq 1$ is the constant from Lemma \ref{Edelta}. Assume that $z_{1} = (x_{1},r_{1}) \in \mathbf{B}_{0}$, $z_{2} = (x_{2},r_{2}) \in \mathbf{B}_{0}$, and $w \in \R^{2}$ are points satisfying
\begin{equation}\label{form82}
  w = (w_{1},w_{2}) \in \Gamma^{\epsilon}(z_{1}) \cap \Gamma^{\epsilon}(z_{2}), \quad t \leq |z_{1} - z_{2}| \leq 2t, \quad \text{and} \quad \Delta_{J}(z_{1} - z_{2}) \leq \epsilon.
\end{equation}
Note that $\Delta'(z_{1} - z_{2}) \leq 2\Delta_{J}(z_{1} - z_{2}) \leq 2\epsilon \leq |z_{1} - z_{2}|/C$ implies $|x_{1} - x_{2}| \sim t$. The heuristic meaning of \eqref{form82} is that the curves $\Gamma(z_{1})$ and $\Gamma(z_{2})$ intersect fairly tangentially at $w \in \R^{2}$, and by \eqref{form80} and the discussion at the end of Section \ref{geomInterpretation}, the same is true for the circles $S(x_{1},r_{2})$ and $S(x_{2},r_{2})$. How is the spatial \textbf{location} of the tangency between $S(x_{1},r_{1})$ and $S(x_{2},r_{2})$ related to $w$? The following lemma answers this question: there are at most constant many rectangles satisfying \eqref{form83}, so the location of tangency, at scale $\epsilon$, between $S(x_{1},r_{1})$ and $S(x_{2},r_{2})$ is roughly determined by the first coordinate of any point in the intersection $\Gamma^{\epsilon}(z_{1}) \cap \Gamma^{\epsilon}(z_{2})$.

\begin{lemma}\label{sineCurvesToCircles}
  Suppose that $z_{1} = (x_{1},r_{1}) \in \mathbf{B}_{0}$, $z_{2} = (x_{2},r_{2}) \in \mathbf{B}_{0}$, and $w \in \R^{2}$ satisfy \eqref{form82}, with $t \geq 2C\epsilon$. Then both circles $S(x_{1},r_{1})$ and $S(x_{2},r_{2})$ are tangent to an $(\epsilon,t)$-rectangle $R$ with 
  \begin{equation}\label{form83}
    R \subset S^{C_{0}\epsilon}(x_{1},r_{1}) \cap B(x_{1} + r_{1}(\cos w_{1},\sin w_{1}),C\sqrt{\epsilon/t}).
  \end{equation}
\end{lemma}

\begin{proof}
  Since $w = (w_1,w_2) \in \Gamma^{\epsilon}(z_1) \cap \Gamma^{\epsilon}(z_2)$, we trivially have $w_{1} \in J/2$ and 
  \begin{displaymath} |\gamma(w_{1}) \cdot (z_{1} - z_{2})| = |\rho_{w_{1}}(z_{1}) - \rho_{w_{1}}(z_{2})| \leq |\rho_{w_{1}}(z_{1}) - w_{2}| + |w_{2} - \rho_{w_{1}}(z_{2})| \leq 2\epsilon, \end{displaymath}
  and, therefore,
  \begin{displaymath}
    w_{1} \in E_{2\epsilon}(z_{1} - z_{2}) = \{\theta \in J/2 : |\gamma(\theta) \cdot (z_{1} - z_{2})| \le 2\epsilon\}.
  \end{displaymath}
  By Lemma \ref{Edelta}, the set $E_{2\epsilon}(z_{1} - z_{2})$ is contained in a single interval of length at most a constant times $\sqrt{\epsilon/t}$ around a certain point $\theta_{0} \in 2J$ with $|\pi_{V_{\theta_{0}}}(z_{1} - z_{2})| \lesssim \Delta_J(z_1-z_2)$. In particular,
  \begin{equation}\label{form70}
    |w_{1} - \theta_{0}| \lesssim \sqrt{\epsilon/t}.
  \end{equation}
  By Lemma \ref{circleLemma}(1), the intersection
  \begin{equation} \label{inclusion}
    S^{C_{0}\epsilon}(x_{1},r_{1}) \cap S^{C_{0}\epsilon}(x_{2},r_{2})
  \end{equation}
  is contained in a disc centred at
  \begin{displaymath}
    \zeta(x_{1},x_{2}) = x_{1} + r_{1}\sgn(r_{1} - r_{2}) \frac{x_{2} - x_{1}}{|x_{2} - x_{1}|} =: x_{1} + r_{1}e(z_{1},z_{2}).
  \end{displaymath}
  and radius at most a constant times $\sqrt{\epsilon/t}$.

  Now, we claim that 
  \begin{equation}\label{form81}
    |(\cos \theta_{0},\sin \theta_{0}) - e(z_{1},z_{2})| \lesssim \frac{\epsilon}{t},
  \end{equation}
  so that, by \eqref{form70},
  \begin{equation}\label{form84}
    |(\cos w_{1},\sin w_{1}) - e(z_{1},z_{2})| \lesssim \sqrt{\frac{\epsilon}{t}}.
  \end{equation}
  Start by recalling from Section \ref{geomInterpretation} that $|\pi_{V_{\theta_{0}}}(z_{1} - z_{2})| \lesssim \Delta_J(z_1-z_2) \leq \epsilon$ implies $\dist(z_{1} - z_{2},\ell_{\theta_{0}}) \lesssim \epsilon$, so we may find $s \in \R$ such that
  \begin{displaymath}
    |(z_{1} - z_{2}) - s(\cos \theta_{0},\sin \theta_{0},-1)| \lesssim \epsilon.
  \end{displaymath}
  It follows that
  \begin{equation}\label{form41a}
    |(x_{1} - x_{2}) - s(\cos \theta_{0},\sin \theta_{0})| \lesssim \epsilon \quad \text{and} \quad |(r_{1} - r_{2}) - s| \lesssim \epsilon.
  \end{equation}
  Abbreviate $\sigma := \sgn(r_{1} - r_{2})$ and $e := e(z_{1},z_{2})$. Then,
  \begin{align*}
    |e - (\cos \theta_{0},\sin \theta_{0})| \leq \left|\frac{x_{2} - x_{1}}{|x_{2} - x_{1}|} - \frac{s(\cos \theta_{0},\sin \theta_{0})}{|x_{2} - x_{1}|} \right|
    + \left| \frac{\sigma s(\cos \theta_{0},\sin \theta_{0})}{|x_{2} - x_{1}|} - (\cos \theta_{0},\sin \theta_{0}) \right|
  \end{align*} 
  Using \eqref{form41a} and the fact that $|x_{1} - x_{2}| \sim t$ (see the discussion after \eqref{form82}), the first term in the right-hand side of the above inequality is bounded by a constant times $\epsilon/t$. The second term admits the same estimate, using \eqref{form41a}:
  \begin{displaymath}
    \left|\frac{\sigma s}{|x_{2} - x_{1}|} -1 \right| \lesssim \frac{|\sigma s - |x_{2} - x_{1}||}{t} \leq \frac{|s-(r_{1} - r_{2})|}{t} + \frac{\Delta'(z_{2} - z_{1})}{t} \lesssim \frac{\epsilon}{t}.
  \end{displaymath}
  This proves \eqref{form81} and hence \eqref{form84}.

  Finally, by Lemma \ref{incidenceTangency}, both circles $S(x_{1},r_{1})$ and $S(x_{2},r_{2})$ are tangent to a certain $(\epsilon,t)$-rectangle $R$, which, by the definition of tangency, the inclusion of \eqref{inclusion}, and \eqref{form84}, means that
  \begin{align*}
    R &\subset S^{C_{0}\epsilon}(x_{1},r_{1}) \cap S^{C_{0}\epsilon}(x_{2},r_{2})\\
    &\subset B(x_{1} + r_{1}e(x_{1},x_{2}),C\sqrt{\epsilon/t})\\
    &\subset B(x_{1} + r_{1}(\cos(w_{1}),\sin (w_{1})),C\sqrt{\epsilon/t}).
  \end{align*}
  This completes the proof of the lemma.
\end{proof}

\section{A measure-theoretic variant of Wolff's incidence bound for tangencies}

One of the main technical innovations in T. Wolff's paper \cite{Wo3} is Lemma 1.4. It bounds the number of incomparable $(\delta,t)$-rectangles, which are tangent to a family of circles. To make the statement precise, we recall some definitions from \cite{Wo3}:

\begin{definition}[Bipartite sets]
  Let $t > 0$. A subset of $\mathbf{B}_{0}$ (recall Definition \ref{B0}) is called \emph{$t$-bipartite}, if it can be written as a disjoint union $\mathcal{W} \cup \mathcal{B}$ ("white" and "black" points) with
  \begin{equation*}
    \max\{\diam(\mathcal{B}),\diam(\mathcal{W})\} \leq t \leq \dist(\mathcal{W},\mathcal{B}) \quad \text{and} \quad \diam(\mathcal{W} \cup \mathcal{B}) \leq 100t.
  \end{equation*}
\end{definition}

We will make an attempt to denote finite $t$-bipartite sets by $\mathcal{W} \cup \mathcal{B}$, and infinite ones by $W \cup B$. It should not cause confusion that $B$ is also a common letter for a ball (the only concrete black set is defined below \eqref{ratio}, and it is in fact an annulus).

\begin{definition}[Type]\label{def:type}
  Assume that $W \cup B \subset \mathbf{B}_{0}$ is a $t$-bipartite set, let $\mu$ be a finite measure on $\R^{3}$, and let $\mathfrak{m},\mathfrak{n} > 0$ be positive real numbers. A $(\delta,t)$-rectangle $R \subset \R^{2}$ is of \emph{type} $\geq \m$ with respect to $\mu,W$ if
  \begin{displaymath}
    \mu(\{(x,r) \in W : S(x,r) \text{ is tangent to } R\}) \geq \m.
  \end{displaymath}
Similarly, $R$ is of type $\geq \n$ with respect to $\mu,B$ if
  \begin{displaymath}
    \mu(\{(x,r) \in B : S(x,r) \text{ is tangent to } R\}) \geq \n.
  \end{displaymath}
  We also define that $R$ is of type $(\geq \m,\geq \n)$ with respect to $\mu,W,B$ if $R$ satisfies both of the requirements above simultaneously. We often omit writing "with respect to $\mu,W,B$", if these parameters are clear from the context. 
\end{definition}

\begin{lemma}[\mbox{\cite[Lemma 1.4]{Wo3}}]\label{mainWolff}
  Let $0 < t < 1$, $0 < \delta \leq ct$, where $c > 0$ is a small absolute constant, let $\m,\n \in \N$, let $\mathcal{W} \cup \mathcal{B}$ be a finite $t$-bipartite set, and let $\mu := \mathcal{H}^{0}|_{\mathcal{W} \cup \mathcal{B}}$. If $\epsilon > 0$, then there is a constant $C_{\epsilon} \geq 1$ such that the cardinality of any collection of pairwise incomparable $(\delta,t)$-rectangles of type $(\geq \m,\geq \n)$ with respect to $\mu,\mathcal{W},\mathcal{B}$ is bounded by
  \begin{displaymath}
    C_{\epsilon}(|\mathcal{W}||\mathcal{B}|)^{\epsilon} \left( \left(\frac{|\mathcal{W}||\mathcal{B}|}{\m\n} \right)^{3/4} + \frac{|\mathcal{W}|}{\m} + \frac{|\mathcal{B}|}{\n} \right).
  \end{displaymath}
\end{lemma}

The purpose of this section is to deduce a variant of Wolff's lemma for arbitrary finite measures; the proof is a straightforward reduction to Lemma \ref{mainWolff}. 

\begin{lemma}\label{wolffMeasures}
 Let $0 < t < 1$, $0 < \delta \leq ct$, where $c > 0$ is a small absolute constant, let $W \cup B \subset \mathbf{B}_{0}$ be a $t$-bipartite set, and let $\mu$ be a probability measure on $\R^{3}$, and let $\m,\n \in (0,1]$. For $\epsilon > 0$, there exists a constant $C_{\epsilon} \geq 1$ such that the cardinality of any set of pairwise incomparable $(\delta,t)$-rectangles of type $(\geq \m, \geq \n)$ is bounded by
  \begin{equation}\label{form47}
    C_{\epsilon}(\m\n\delta)^{-\epsilon} \left( \left(\frac{\mu(W)\mu(B)}{\m\n} \right)^{3/4} + \frac{\mu(W)}{\m} + \frac{\mu(B)}{\n} \right).
  \end{equation}
\end{lemma}

\begin{proof}
  Assume without loss of generality that $\delta > 0$ is a small dyadic number, and denote by $\calD_{\delta}$ the dyadic cubes in $\R^{3}$ of side-length $\delta$. For $i,j \geq 0$, let $\calD^{W}_{i} := \{Q \in \calD_\delta : 2^{-i - 1} \leq \mu(Q \cap W) \leq 2^{-i}\}$ and $\calD^{B}_{j} := \{Q \in \calD_\delta : 2^{-j - 1} \leq \mu(Q \cap B) \leq 2^{-j}\}$. Note that if $R$ is a $(\delta,t)$-rectangle $C_{0}$-tangent to any circle $S(x,r)$ with $(x,r) \in Q$, with $Q \cap \mathbf{B}_{0} \neq \emptyset$, then $R$ is $2C_{0}$-tangent to the circle $S(x_{Q},r_{Q})$, where $(x_{Q},r_{Q})$ is the midpoint of $Q$. 

  Let $\calR$ be a maximal collection of incomparable $(\delta,t)$-rectangles of type $(\geq \m,\geq \n)$. Then, for $R \in \calR$, there is a set 
  \begin{displaymath}
    W_{R} := \{(x,r) \in W : S(x,r) \text{ is tangent to } R\} \subset W
  \end{displaymath}
  with $\mu(W_{R}) \geq \m$ such that $R$ is tangent to every circle from $W_{R}$. Let $W_{R}^{i}$ be the set $W_{R}$ intersected with the union of the cubes in $\calD_{i}^{W}$. Define $B_{R}$ and $B_{R}^{j}$ similarly, using $\calD_{j}^{B}$. Then, there exist $i_{R},j_{R} \geq 1$ such that
  \begin{displaymath}
    \mu(W_{R}^{i_{R}}) \gtrsim \frac{\m}{i_{R}^{2}} \quad \text{and} \quad \mu(B_{R}^{j_{R}}) \gtrsim \frac{\n}{j_{R}^{2}}.
  \end{displaymath}
  Since $W \subset \mathbf{B}_{0}$, the total $\mu$-measure of cubes in $\calD^{W}_{i}$ is at most a constant times $\delta^{-3}2^{-i}$. Therefore,
  \begin{displaymath}
    \frac{\m}{i_{R}^{2}} \lesssim \mu(W_{R}^{i_{R}}) \lesssim \delta^{-3}2^{-i_{R}}.
  \end{displaymath}
 Given $\epsilon > 0$, the inequality above implies $i_{R} \leq C_{\epsilon}\delta^{-3\epsilon}\m^{-\epsilon}$ for a constant $C_{\epsilon} \geq 1$ depending only on $\epsilon$. The same reasoning applies to $j_{R}$, with $\m$ replaced by $\n$. Now, for each $i \in \{1,\ldots,C_{\epsilon}\delta^{-3\epsilon}\m^{-\epsilon}\}$ and $j \in \{1,\ldots,C_{\epsilon}\delta^{-3\epsilon}\n^{-\epsilon}\}$, we define
  \begin{displaymath}
    \mathcal{R}^{(i,j)} := \{R \in \mathcal{R} : i_{R} = i \text{ and } j_{R} = j\}.
  \end{displaymath}
  Then we pick $(i,j)$ such that $\mathcal{R}^{(i,j)} =: \calR'$ is the largest to obtain
  \begin{equation}\label{form46}
    |\calR'| \gtrsim_{\epsilon} |\calR| \cdot \delta^{6\epsilon} \m^{\epsilon} \n^{\epsilon}
  \end{equation}
  With these values of $i,j$, denote by $\mathcal{W}^{i}$ and $\mathcal{B}^{j}$ the midpoints of the cubes in $\calD_{i}^{W}$ and $\calD_{j}^{B}$, respectively.

  Fix a rectangle $R \in \calR'$. By the definition, and recalling that $i \in \{1,\ldots,C_{\epsilon}\delta^{-3\epsilon}m^{-\epsilon}\}$,
  \begin{displaymath}
    \mu(W^{i}_{R}) \gtrsim \frac{m}{i^{2}} \gtrsim_{\epsilon} \m^{1 + 2\epsilon}\delta^{6\epsilon}.
  \end{displaymath}
  Since $\mu(Q \cap W^{i}_{R}) \leq \mu(Q \cap W) \sim 2^{-i}$ for all $Q \in \calD_{i}^{W}$, we infer that at least $\overline{\m} \gtrsim_{\epsilon} \max\{1,2^{i}\m^{1 + 2\epsilon}\delta^{6\epsilon}\}$ cubes $Q \in \calD_{i}^{W}$ intersect $W_{R}^{i}$, where $\overline{\m} \in \N$. As discussed above, this means that $R$ is $2C_{0}$-tangent to $S(x_{Q},r_{Q})$ for each of these cubes $Q$; note that $(x_{Q},r_{Q}) \in \calW^{i}$ for these $Q$, by the definition of $\calW^{i}$. The same reasoning applies to $B_{R}^{j}$, and the conclusion is that $R$ is of type 
  \begin{displaymath}
    (\geq \overline{\m},\geq \overline{\n}) \quad \text{for} \quad \overline{\m} \gtrsim_{\epsilon} \max\{1,2^{i}m^{1 + 2\epsilon}\delta^{6\epsilon}\} \text{ and } \overline{\n} \gtrsim_{\epsilon} \max\{1,2^{j}n^{1 + 2\epsilon}\delta^{6\epsilon}\},
  \end{displaymath}
  with respect to the $t$-bipartite family $\calW^{i} \cup \calB^{j}$ and the counting measure. We first assume that $2^{i}\m^{1 + 2\epsilon}\delta^{6\epsilon} \geq 1$ and $2^{j}\n^{1 + 2\epsilon}\delta^{6\epsilon} \geq 1$. Then, using Lemma \ref{mainWolff}, we infer that
  \begin{equation}\label{form45}
    |\calR'| \lesssim_{\epsilon} (|\calW^{i}||\calB^{j}|)^{\epsilon} \left( \left(\frac{|\calW^{i}||\calB^{j}|}{2^{i + j}(\m\n)^{1 + 2\epsilon}\delta^{12\epsilon}} \right)^{3/4} + \frac{|\calW^{i}|}{2^{i}\m^{1 + 2\epsilon}\delta^{6\epsilon}} + \frac{|\calB^{j}|}{2^{j}\n^{1 + 2\epsilon}\delta^{6\epsilon}} \right).
  \end{equation}
  To make the right hand side of \eqref{form45} look more like the right hand side of \eqref{form47}, note that
  \begin{displaymath}
    |\calW^{i}| = \sum_{(x_{Q},r_{Q}) \in \calW^{i}} \frac{2^{-i}}{2^{-i}} \lesssim 2^{i} \sum_{(x_{Q},r_{Q}) \in \calW^{i}} \mu(Q \cap W) \lesssim 2^{i}\mu(W) \quad \text{and} \quad |\calB^{j}| \lesssim 2^{j}\mu(B).
  \end{displaymath}
  Combined with \eqref{form46} and \eqref{form45}, this completes the proof of \eqref{form47} in the case $2^{i}\m^{1 + 2\epsilon}\delta^{6\epsilon} \geq 1$ and $2^{j}\n^{1 + 2\epsilon}\delta^{6\epsilon} \geq 1$. Let us consider the case $2^{i}\m^{1 + 2\epsilon}\delta^{6\epsilon} < 1$ and $2^{j}\n^{1 + 2\epsilon}\delta^{6\epsilon} < 1$, leaving the intermediate cases to the reader. Then $\overline{\m} = 1 = \overline{\n}$, and instead of \eqref{form45}, we may infer from Lemma \ref{mainWolff} that
 \begin{displaymath}  |\calR'| \lesssim_{\epsilon} (|\calW^{i}||\calB^{j}|)^{\epsilon} \left( \left(|\calW^{i}||\calB^{j}| \right)^{3/4} + |\calW^{i}| + |\calB^{j}| \right). \end{displaymath} 
 To conclude the proof of \eqref{form47} from here, one then uses the inequalities $|\mathcal{W}^{i}| \lesssim 2^{i}\mu(W) \leq \mu(W)\m^{-1 - 2\epsilon}\delta^{-6\epsilon}$ and $|\mathcal{B}^{j}| \lesssim 2^{j}\mu(B) \leq \mu(B)\n^{-1 - 2\epsilon}\delta^{-6\epsilon}$, and also $\mu(W),\mu(B) \leq 1$. \end{proof}

\section{A measure-theoretic variant of Schlag's lemma for circles}

Lemma \ref{incidenceCircles} below is the main tool in the proof of Theorem \ref{circleUnion} about unions of circles. It is a continuous version of W. Schlag's weak type inequality in \cite[Lemma 8]{Sc}. The proof follows the same pattern, but the statement is a bit stronger (involving measures, not finite sets), and the argument is a bit simpler; for example, we can omit the case distinction between "$\delta \sim \epsilon$" and "$\delta \gg \epsilon$" altogether, and also the selection of a random $\epsilon$-separated subset; see the proof in \cite{Sc}. 

Aside from being crucial in the proof Theorem \ref{circleUnion}, Lemma \ref{incidenceCircles} is also used within the proof of Lemma \ref{incidenceSineCurves}, which is finally the key ingredient in the proof of the main result, Theorem \ref{main}. 

Recall that $S^{\delta}(x,r)$ stands for the $\delta$-neighbourhood of the planar circle $S(x,r)$ and
\begin{displaymath}
  \Delta'(z) = ||x| - |r||, \qquad z = (x,r).
\end{displaymath}
Given a finite measure $\mu$ on $\R^{3}$ and $\delta > 0$, define the following \emph{multiplicity} function $m_{\delta}^{\mu} \colon \R^{2} \to [0,\mu(\R^{3})]$:
\begin{equation}\label{multcircles}
  m^{\mu}_{\delta}(w) = \mu(\{z' \in \R^{3} : w \in S^{\delta}(z')\}).
\end{equation}

\begin{lemma}\label{incidenceCircles}
  Fix $s \in (0,1]$, $\delta > 0$, $\eta > 0$, $\mathbf{C} \geq 1$, and $A \geq C_{\eta,\mathbf{C},s} \cdot \delta^{-\eta}$, where $C_{\eta,\mathbf{C},s} \geq 1$ is a large constant depending only on $\eta,\mathbf{C}$, and $s$. Let $\mu$ be a probability measure on $\R^{3}$ satisfying the Frostman condition $\mu(B(z,r)) \leq \mathbf{C} r^{s}$ for all $z \in \R^{3}$ and $r > 0$, and with $K := \spt \, \mu \subset \mathbf{B}_{0}$.  Then, for $\lambda \in (0,1]$, there is a set $G(A,\delta,\lambda) \subset K$ with 
  \begin{displaymath}
    \mu(K \setminus G(A,\delta,\lambda)) \leq A^{-s/3}
  \end{displaymath}
  such that the following holds for all $z \in G(A,\delta,\lambda)$:
  \begin{displaymath}
    |S^{\delta}(z) \cap \{w : m^{\mu}_{\delta}(w) \geq A^{s}\lambda^{-2s} \delta^{s}\}| \leq \lambda |S^{\delta}(z)|.
  \end{displaymath}
\end{lemma}

\begin{proof} We start by remarking that the lemma is trivial for all $\delta \gtrsim_{\eta,\mathbf{C},s} 1$ and for all $\lambda \in (0,1]$. Indeed, in this case we may choose the lower bound $C_{\eta,\mathbf{C},s}$ for $A$ so large that $A^{s}\lambda^{-2s}\delta^{s} > 1$. This has the effect that $m_{\delta}^{\mu}(w) \geq A^{s}\lambda^{-2s}\delta^{s}$ can never hold, since $m_{\delta}^{\mu}(w) \leq \mu(\R^{3}) = 1$. So, in the sequel, we may assume that $\delta$ is small, in a manner depending on $\eta,\mathbf{C},s$. For similar reasons, we may assume that $A \leq \delta^{-1}$: otherwise $m_{\delta}^{\mu} \leq 1 < A^{s}\lambda^{-2s}\delta^{s}$. 

  Assume then to the contrary there exists a dyadic number $\delta \in 2^{-\N}$, and a number 
  \begin{equation}\label{form107}
    m \geq A^{s}\lambda^{-2s}\delta^{s}
  \end{equation}
  such that 
  \begin{equation}\label{form49a}
    |S^{\delta}(z) \cap \{w : m_{\delta}^{\mu}(w) \geq m\}| \geq \lambda |S^{\delta}(z)|
  \end{equation}
  for every $z \in D \subset K$, where
  \begin{equation*}
    \mu(D) > A^{-s/3}.
  \end{equation*}
  This will result in a contradiction, if $C_{\eta,\mathbf{C},s}$ in the assumption $A \geq C_{\eta,\mathbf{C},s} \cdot \delta^{-\eta}$ is sufficiently large. For the purposes of induction, we assume that $\delta \in 2^{-\N}$ is the largest dyadic number failing the statement of the lemma for some $\lambda \in (0,1]$ and $A \geq C_{\eta,\mathbf{C},s} \cdot \delta^{-\eta}$ (as we already observed in the first paragraph of the proof, the statement is trivial for $\delta \gtrsim_{\mathbf{C},\eta,s} 1$, for all $\lambda \in (0,1]$, and all $A \geq C_{\eta,\mathbf{C},s}$, so the "base case" of the induction is valid.)
  
  For $z \in \R^{3}$ and dyadic numbers $\epsilon,t \in [\delta,1]$, define
  \begin{equation*}
    K_{\epsilon,t}(z) := \{z' \in K \colon S^{\delta}(z) \cap S^{\delta}(z') \neq \emptyset, \: t \leq |z - z'| < 2t, \text{ and }  \epsilon \leq \Delta'(z - z') < 2\epsilon\}.
  \end{equation*}
  The case $\epsilon = \delta$ is a little special: there we modify the definition so that the two-sided inequality $\epsilon \leq \Delta'(z - z') < 2\epsilon$ is replaced by simply $\Delta'(z - z') < 2\epsilon = 2\delta$. Now, define the restricted multiplicity function
  \begin{displaymath}
    m^{\mu}_{\delta}(w | U) := \mu(\{z' \in U : w \in S^{\delta}(z')\}), \qquad U \subset \R^{3}.
  \end{displaymath}
  Applied with $U := K_{\epsilon,t}(z)$, as we will do in a moment, this multiplicity function only takes into accounts those $z'$, which are at distance $t$ from, and $\epsilon$-tangent to, $z$. If $z \in D$ is fixed, $w \in \R^{2}$ is such that $m^{\mu}_{\delta}(w) \geq m$ (as in \eqref{form49a}), and $C \geq 1$ is a large enough absolute constant (to be determined later), then we consider the inequality
  \begin{align*}
    C_{\eta,\mathbf{C},s}^{s}\delta^{s} \stackrel{\eqref{form107}}{\leq} m \leq m^{\mu}_{\delta}(w) &\leq \mu(B(z,C\delta)) + \mathop{\sum_{t \in [C\delta,1]}}_{\epsilon \in [\delta,1]} m^{\mu}_{\delta}(w | K_{\epsilon,t}(z))\\
    &\leq \mathbf{C}(C\delta)^{s} + \mathop{\sum_{t \in [C\delta,1]}}_{\epsilon \in [\delta,1]} m^{\mu}_{\delta}(w | K_{\epsilon,t}(z)),
  \end{align*} 
  where $\epsilon$ and $t$ only run over dyadic values. We take $C_{\eta,\mathbf{C},s} \geq 1$ (at least) so large that $\mathbf{C}C^{s} \leq \tfrac{1}{2}C^{s}_{\eta,\mathbf{C},s}$. Then the second term in the display above must dominate the left hand side. This implies (after a few rounds of pigeonholing) that there exist dyadic numbers $\epsilon \in [\delta,1]$ and $t \in [C\delta,1]$, $m \lessapprox \bar{m} \leq m$ and $\lambda \lessapprox \bar{\lambda} \leq \lambda$, and a subset $\bar{D} \subset D$ with $\mu(\bar{D}) \gtrapprox \mu(D)$, such that the following holds for all $z \in \bar{D}$:
  \begin{equation}\label{barDa}
    |H^{\delta}(z)| := |S^{\delta}(z) \cap \{w :  m^{\mu}_{\delta}(w | K_{\epsilon,t}(z)) \geq \bar{m}\}| \geq \bar{\lambda}|S^{\delta}(z)|.
  \end{equation}
  We recall from Section \ref{s:notation} that the notation $C_{1} \lessapprox C_{2}$ signifies an inequality of the form
  \begin{displaymath}
    C_{1} \leq C(\log(1/\delta))^{C}C_{2}
  \end{displaymath}
  for some absolute constant $C \geq 1$. Furthermore, by $C_1 \approx C_2$ was the same as $C_1 \lessapprox C_2 \lessapprox C_1$. For the rest of the proof, the numbers $t$ and $\epsilon$ will be the ones we found above, and we abbreviate 
  \begin{equation}\label{form109} K_{\epsilon,t}(z) =: K(z). \end{equation}

  We now make a brief heuristic digression. By the preceding discussion, we have found that a large fraction of the "high density" part of $S^{\delta}(z)$, for $z \in \bar{D}$, is caused by points $z'$, which are at roughly distance $t \gg \delta$ from $z$, and moreover the tangency between $S(z)$ and $S(z')$ is roughly constant, namely $\epsilon$. This means that the circles $S(z)$ and $S(z')$ are $\epsilon$-incident, and hence they are tangent to an $(\epsilon,t)$-rectangle by Lemma \ref{incidenceTangency}. To complete the proof, it suffices to count, just how many incomparable $(\epsilon,t)$-rectangles we can find this way (by varying $z \in \bar{D}$ and $z' \in K_{\epsilon,t}(z)$), and then compare the figure with the upper bound given by Lemma \ref{wolffMeasures} to reach a contradiction. If $\epsilon = \delta$, this is straightforward, but if $\epsilon \gg \delta$, an additional geometric argument is needed: in brief, we will show that a perfect analogue of \eqref{barDa} also holds at scale $\epsilon$, for every $z \in \bar{D}$: see \eqref{form64} below, and note in particular that \eqref{form64} and \eqref{barDa} are essentially the same, if $\epsilon = \delta$. In a sense, the argument leading to \eqref{form64} is just a complicated way of saying that "$\epsilon = \delta$ without loss of generality".

  We continue with the proof. Fix $z \in \bar{D}$, and recall that \eqref{barDa} holds. We claim that there exists a dyadic number $\nu = \nu(z) \in \{1,\ldots,\epsilon/\delta\}$, and an absolute constant $C_{1} \geq 2$, such that
  \begin{equation}\label{form62}
    |S^{\epsilon}(z) \cap \{w : m^{\mu}_{C_{1}\epsilon}(w|K(z)) \gtrsim \nu\bar{m}\}| \gtrapprox \frac{\bar{\lambda}\epsilon}{\nu \delta} \cdot |S^{\epsilon}(z)|.
  \end{equation}
  To see this, we need to recall the geometric fact from Lemma \ref{circleLemma} that if $z,z' \in \R^{3}$ and with $|z - z'| \sim t$ and $\Delta'(z - z') \sim \epsilon$, then $S^{\delta}(z) \cap S^{\delta}(z')$ can be covered by two $\delta$-neighbourhoods of arcs on $S(z)$, each of diameter at most a constant times $C\delta/\sqrt{\epsilon t}$. (If $\epsilon = \delta$, then we only have the one-sided information $\Delta'(z - z') < 2\epsilon = 2\delta$, but the geometric statement above remains valid, even with "two arcs" replaced by "one arc".) Motivated by this, we first divide $S(z)$ into \emph{short arcs} $J_{j}(z)$ of length $C\delta/\sqrt{\epsilon t}$. We write $J_{j}^{\delta}(z)$ for the $\delta$-neighbourhood of $J_{j}(z)$. Since $|J^{\delta}_{j}(z)| \lesssim \delta^{2}/\sqrt{\epsilon t}$, we may, by \eqref{barDa}, find at least a constant times
  \begin{displaymath}
    \frac{\bar{\lambda}|S^{\delta}(z)|}{\delta^{2}/\sqrt{\epsilon t}} \sim \frac{\bar{\lambda}\sqrt{\epsilon t}}{\delta}
  \end{displaymath}
  indices $j$ such that $H^{\delta}(z) \cap J_{j}^{\delta}(z) \neq \emptyset$. Denote these indices by $\calJ(z)$, and for each $j \in \calJ(z)$, pick a point $w_{j} \in H^{\delta}(z) \cap J^{\delta}_{j}(z)$. Thus $m^{\mu}_{\delta}(w_{j}|K(z)) \geq \bar{m}$ for $j \in \calJ(z)$. Throw away at most half of the points $w_{j}$ to ensure that $|w_{i} - w_{j}| \geq C\delta/\sqrt{\epsilon t}$ for all $i \neq j$. Then, the sets
  \begin{displaymath}
    K_{j}(z) := \{z' \in K(z) : w_{j} \in S^{\delta}(z) \cap S^{\delta}(z')\}
  \end{displaymath}
  with 
  \begin{equation}\label{form86}
    \mu(K_{j}(z)) = m_{\delta}^{\mu}(w_{j}|K(z)) \geq \bar{m}
  \end{equation}
  have bounded overlap:
  \begin{equation}\label{bdOverlap}
    \sum_{j \in \calJ(z)} \chi_{K_{j}(z)}(z') \leq 2, \qquad z' \in K(z).
  \end{equation}
  Indeed, if $z' \in K_{j}(z)$, then $w_{j} \in S^{\delta}(z) \cap S^{\delta}(z')$, which implies $w_{j}$ has to lie in one of the at most two sets of diameter at most $C\delta/\sqrt{\epsilon t}$ covering the intersection $S^{\delta}(z) \cap S^{\delta}(z')$. By the separation of the points $w_{j}$, this can happen for at most two values of $j$.

  Next, we group the points $w_{j}$ inside sets of somewhat larger diameter (than $J_{j}(z)$). To this end, divide $S(z)$ into \emph{long arcs} $I_{i}(z)$ of length $C\sqrt{\epsilon/t}$. By adjusting the lengths of both long and short arcs slightly, we may assume that the long arcs $I_{i}(z)$ are sub-divided further into
  \begin{displaymath}
    \frac{\calH^{1}(I_{i}(z))}{\calH^{1}(J_{j}(z))} \leq \frac{\sqrt{\epsilon/t}}{\delta/\sqrt{\epsilon t}} = \frac{\epsilon}{\delta}
  \end{displaymath}
  short arcs $J_{j}(z)$. For each long arc $I_{i}(z)$, write
  \begin{displaymath}
    k(i) := \card \{j \in \calJ(z) : w_{j} \in I^{\delta}_{i}(z)\},
  \end{displaymath} 
  where $I_{i}^{\delta}(z)$ is the $\delta$-neighbourhood of $I_{i}(z)$. Since $0 \leq k(i) \leq \epsilon/\delta$, there is a dyadic number $\nu = \nu(z) \in \{1,\ldots,\epsilon/\delta\}$ such that $\gtrapprox |\calJ(z)| \gtrsim \bar{\lambda}\sqrt{\epsilon t}/\delta$ points $w_{j}$ are contained in the union of the sets $I^{\delta}_{i}(z)$ with $\nu \leq k(i) \leq 2\nu$. Denote the indices of these sets $I^{\delta}_{i}(z)$ by $\calI(z)$. Thus, if $i \in \calI(z)$, then 
  \begin{equation}\label{form87}
    \card\{j \in \calJ(z) : w_{j} \in I_{i}^{\delta}(z)\} \sim \nu.
  \end{equation}
  Since there are $\gtrapprox |\calJ(z)|$ points in total, we conclude that
  \begin{equation}\label{form88}
    |\calI(z)| \gtrapprox \frac{|\calJ(z)|}{\nu} \gtrsim \frac{\bar{\lambda} \sqrt{\epsilon t}}{\nu \delta}.
  \end{equation}

  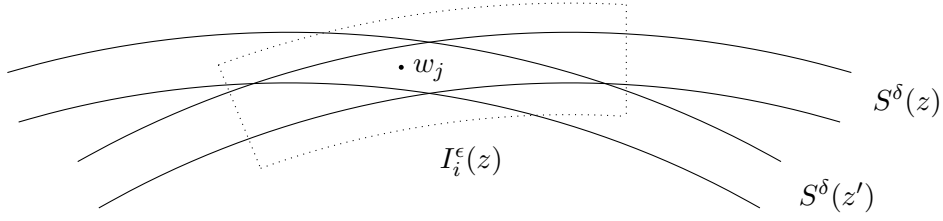
\begin{figure}[t!]
  \begin{center}
    \begin{tikzpicture}[scale=1.85]
      \draw [domain=0:5.5,smooth,variable=\x] plot ({\x},{sqrt{(50-(\x-2)*(\x-2))}});
      \draw [domain=0.08:5.35,smooth,variable=\x] plot ({\x},{sqrt{(45-(\x-2)*(\x-2))}});
      \node[align=center] at (5.9,5.9) {$S^\delta(z')$};

      \draw [domain=0.5:6,smooth,variable=\x] plot ({\x},{sqrt{(50-(\x-4)*(\x-4))}});
      \draw [domain=0.65:5.92,smooth,variable=\x] plot ({\x},{sqrt{(45-(\x-4)*(\x-4))}});
      \node[align=center] at (6.4,6.58) {$S^\delta(z)$};

      \fill (2.8,6.82) circle [radius=0.5pt];
      \node[align=center] at (3,6.8) {$w_j$};

      \draw [domain=1.5:4.4,dotted,smooth,variable=\x] plot ({\x},{sqrt{(53-(\x-4)*(\x-4))}});
      \draw [domain=1.8:4.4,dotted,smooth,variable=\x] plot ({\x},{sqrt{(42-(\x-4)*(\x-4))}});
      \draw [dotted] (1.5,6.8374) -- (1.8,6.096);
      \draw [dotted] (4.4,7.2691) -- (4.4,6.4684);
      \node[align=center] at (3.3,6.15) {$I_i^\epsilon(z)$};
    \end{tikzpicture}
    \caption{An illustration for the proof of Lemma \ref{incidenceCircles}.}
    \label{fig1}
  \end{center}
  \end{figure}

  Fix $i \in \calI(z)$ and $w_{j} \in I_{i}^{\delta}(z)$. We claim that if $z' \in K_{j}(z)$, then 
  \begin{equation}\label{form85}
    I_{i}^{\epsilon}(z) \subset S^{C_{1}\epsilon}(z'),
  \end{equation}
  for some $C_{1} \geq 1$ large enough; see Figure \ref{fig1}. The reason is that $w_{j} \in S^{\delta}(z) \cap S^{\delta}(z')$ and $S(z),S(z')$ are $\epsilon$-incident, so for large enough $C_{1} \geq 1$, they are both $C_{1}$-tangent to the $(\epsilon,t)$-rectangle $I_{i}^{\epsilon}(z) \ni w_{j}$, see Lemma \ref{incidenceTangency} for a similar statement (and its proof in Wolff's paper for more details). Now, for a fixed index $i \in \calI$, and for any $w \in I_{i}^{\epsilon}(z)$, the bounded overlap of the sets $K_{j}(z)$ yields
  \begin{align*}
    m_{C_{1}\epsilon}^{\mu}(w|K(z)) &= \mu(\{z' \in K(z) : w \in S^{C_{1}\epsilon}(z')\})\\
    &\stackrel{\eqref{bdOverlap}}{\geq} \frac{1}{2} \mathop{\sum_{j \in \calJ(z)}}_{w_{j} \in I^{\delta}_{i}(z)} \int_{K_{j}(z)} \chi_{S^{C_{1}\epsilon}(z')}(w) \dd\mu z'\\
    &\stackrel{\eqref{form85}}{=} \mathop{\sum_{j \in \calJ(z)}}_{w_{j} \in I^{\delta}_{i}(z)} \mu(K_{j}(z)) \stackrel{\eqref{form86} \& \eqref{form87}}{\gtrsim} \nu \bar{m}.
  \end{align*}
  Thus, we have proven that whenever $i \in \calI(z)$, the set $I^{\epsilon}_{i}(z) \subset S^{\epsilon}(z)$ is contained in the region where $m_{C_{1}\epsilon}^{\mu}(w|K(z)) \gtrsim \nu \bar{m}$. Recalling \eqref{form88}, this proves that
  \begin{equation}\label{form108}
    |S^{\epsilon}(z) \cap \{w : m_{C_{1}\epsilon}^{\mu}(w|K(z)) \gtrsim \nu \bar{m}\}| \geq |\calI(z)| |I^{\epsilon}_{i}(z)| \stackrel{\eqref{form88}}{\gtrapprox} \frac{\bar{\lambda}\sqrt{\epsilon t}}{\nu \delta} \cdot \sqrt{\epsilon/t} \cdot \epsilon \sim \frac{\bar{\lambda}\epsilon}{\nu \delta} |S^{\epsilon}(z)|,
  \end{equation}
  which is precisely \eqref{form62}. 

  Recall that the dyadic number $\nu = \nu(z) \leq \epsilon/\delta$ still depends on the point $z \in \bar{D}$, but there are only $\lessapprox 1$ possible choices for $\nu(z)$. We replace $\bar{D}$ by a subset of measure $\gtrapprox \mu(\bar{D}) \gtrapprox A^{-s/3}$ to make the choice uniform. Hence, we may assume that \eqref{form62} holds for all $z \in \bar{D}$, for some fixed dyadic number $\nu \in \{1,\ldots,\epsilon/\delta\}$. We now will re-write \eqref{form62} slightly, in such a way that the inequality looks more like \eqref{form49a}, only at scale (roughly) $\epsilon$ instead of $\delta$. For this purpose, we denote
  \begin{equation}\label{form95}
    \lambda_{\epsilon} := \log^{-C}(1/\delta) \frac{\bar{\lambda}\epsilon}{\nu \delta} \quad \text{and} \quad A_{\epsilon} := \log^{C}(1/\delta) \left(\frac{A \epsilon}{\nu \delta} \right) \geq \log^{C}(1/\delta)A,
  \end{equation}
  where $C \geq 1$ is a suitable constant. We may deduce from \eqref{form108} that $(\bar{\lambda} \epsilon)/(\nu\delta) \lessapprox 1$, so $\lambda_{\epsilon} \in (0,1]$ if the constant $C \geq 1$ in the definition above is chosen sufficiently large.
  
 Recall from \eqref{form107}, and the choices $\bar{m} \approx m$ and $\bar{\lambda} \approx \lambda$ just above \eqref{barDa}, that $\bar{m} \gtrapprox A^{s}\bar{\lambda}^{-2s}\delta^{s}$. We define $\m_{\epsilon} := \nu\bar{m}/C$. Then, using the lower bound for $\bar{m}$, and also that $s \in (0,1]$, we record that
  \begin{equation}\label{form63}
    \m_{\epsilon} = \frac{\nu \bar{m}}{C} \gtrapprox \nu A^{s} \bar{\lambda}^{-2s}\delta^{s} \gtrapprox \left(\frac{A \epsilon}{\nu \delta} \right)^{s} \lambda^{-2s}_{\epsilon} \epsilon^{s} \approx A_{\epsilon}^{s}\lambda^{-2s}_{\epsilon}\epsilon^{s}.
  \end{equation}
  Now, if $C$ (in both $\m_{\epsilon} = \nu\bar{m}/C$ and the definition of $\lambda_{\epsilon}$) is large enough, \eqref{form62} implies
  \begin{equation}\label{form64}
    |S^{\epsilon}(z) \cap \{w : m^{\mu}_{C_{1}\epsilon}(w|K(z)) \geq \m_{\epsilon}\}| \geq 2\lambda_{\epsilon} |S^{\epsilon}(z)|, \qquad z \in \bar{D}.
  \end{equation}
  Now \eqref{form63}-\eqref{form64} look like analogues of \eqref{form107}-\eqref{form49a}, only at the (possibly) larger scale $\epsilon$.
   
  Fix a large absolute constant $N \in 2^{\N}$, whose precise value will be determined later, and will only depend on the size of the absolute constant $C_{1} \geq 1$ chosen at \eqref{form85}. Note (using $\nu \leq \epsilon/\delta$ and $\epsilon \geq \delta$) that 
  \begin{displaymath}
    A_{\epsilon} \geq A \geq C_{\eta,\mathbf{C},s} \cdot \delta^{-\eta} \geq C_{\eta,\mathbf{C},s}(N\epsilon)^{-\eta},
  \end{displaymath}
  and we already observed below \eqref{form95} that $\lambda_{\epsilon} \in (0,1]$, so in particular $\lambda_{\epsilon}/(CN) \in (0,1]$. These facts place the induction hypothesis at our disposal, at scale $N\epsilon \geq 2\delta$. Namely, we know that for all points $z \in G := G(A_{\epsilon},N\epsilon,\lambda_{\epsilon}/CN)$, with $\mu(K \setminus G) \leq A_{\epsilon}^{-s/3}$, the following holds:
  \begin{equation}\label{form65}
    |S^{N\epsilon}(z) \cap \{w : m_{N\epsilon}^{\mu}(w) \geq A^{s}_{\epsilon} (\lambda_{\epsilon}/CN)^{-2s}\epsilon^{s}\}| \leq \frac{\lambda_{\epsilon}}{CN} |S^{N\epsilon}(z)| \leq \lambda_{\epsilon}|S^{\epsilon}(z)|.
  \end{equation}
  In particular, since $A_{\epsilon} \geq \log^{C}(1/\delta) A$, and $\mu(\bar{D}) \gtrapprox A^{-s/3}$, the estimate \eqref{form65} holds for at least half of the points $z \in \bar{D}$ (if $C$ is large enough). We restrict attention to this half, so that \eqref{form64}--\eqref{form65} hold simultaneously for all $z \in \bar{D}$. Writing $M_{\epsilon} := A_{\epsilon}^{s}(\lambda_{\epsilon}/CN)^{-2s}\epsilon^{s}$, it follows that
  \begin{equation}\label{form66}
    |S^{\epsilon}(z) \cap \{w : \m_{\epsilon} \leq m^{\mu}_{C_{1}\epsilon}(w|K(z)) \leq m^{\mu}_{N\epsilon}(w) \leq M_{\epsilon}\}| \geq \lambda_{\epsilon} |S^{C\epsilon}(z)|, \qquad z \in \bar{D}.
  \end{equation}
  It should be noted that 
  \begin{equation}\label{form91}
    \m_{\epsilon} \gtrapprox M_{\epsilon}
  \end{equation}
  by \eqref{form63}.

  Now, \eqref{form66} will will give a lower bound for how many circles $S(z)$ are tangent to each other at resolution $\epsilon$. The proof will be completed by comparing this lower bound against the upper bound given by Lemma \ref{wolffMeasures}. For this purpose, we need to extract two sets 
  \begin{displaymath}
    W \subset \bar{D} \quad \text{and} \quad B \subset \R^{3}
  \end{displaymath}
  satisfying the two $t$-bipartite conditions $\max\{\diam(W),\diam(B)\} \leq t \leq \dist(W,B)$ and $\diam(W \cup B) \leq 100t$. We will moreover do this so that 
  \begin{equation}\label{form92}
    \mu(W) \gtrapprox A^{-s/3} \mu(B),
  \end{equation}
  and 
  \begin{equation}\label{form50}
    m_{C_{1}\epsilon}^{\mu}(w|B \cap K(z)) = m_{C_{1}\epsilon}^{\mu}(w|K(z)), \qquad z \in W, \: w \in \R^{2}.
  \end{equation}
  Finding $W$ and $B$ is straightforward. We first cover $\bar{D}$ by $\leq Ct^{-3}$ balls $B(z_{i},t/10)$, such that the balls $B(z_{i},(2 + \tfrac{1}{10})t)$ have bounded overlap. Next, we discard all those balls with $\mu(\bar{D} \cap B(z_{i},t/10)) \leq t^{3}\mu(\bar{D})/(2C)$, and observe that the union of the remaining balls still contains at least half the $\mu$ measure of $\bar{D}$. Next, among the remaining balls, which now all satisfy
  \begin{equation}\label{measureOfW}
    \mu(\bar{D} \cap B(z_{i},t/10)) \geq \frac{t^{3}\mu(\bar{D})}{2C} \gtrapprox A^{-s/3}t^{3},
  \end{equation}
  we set 
  \begin{displaymath} W := \bar{D} \cap B(z_{i},t/10) \end{displaymath}
  for the ball $B(z_{i},t/10)$, which maximises the ratio $\mu(\bar{D} \cap B(z_{i},t/10))/\mu(B(z_{i},(2 + \tfrac{1}{10})t))$. Since the balls $B(z_{i},(2 + \tfrac{1}{10})t)$ have bounded overlap, it follows that
  \begin{equation}\label{ratio}
    \frac{\mu(W)}{\mu(B(z_{i},(2 + \tfrac{1}{10})t))} \gtrsim \mu(\bar{D}) \gtrapprox A^{-s/3}.
  \end{equation}
  Then, we define 
  \begin{displaymath}
    B := B(z_{i},(2 + \tfrac{1}{10})t) \setminus B(z_{i},(1 - \tfrac{1}{10})t) ,
  \end{displaymath}
  and note that $K(z) = B \cap K(z)$ for all $z \in W$, because $z' \in K(z) = K_{\epsilon,t}(z)$ (recall \eqref{form109}) already forces the restriction $t \leq |z - z'| \leq 2t$. Hence, for $w \in \R^{2}$,
  \begin{align*}
    m_{C_{1}\epsilon}^{\mu}(w|B \cap K(z)) &= \mu(\{z' \in B \cap K(z) : w \in S^{C_{1}\epsilon}(z)\})\\
    &= \mu(\{z' \in K(z) : w \in S^{C_{1}\epsilon}(z)\}) = m_{C_{1}\epsilon}^{\mu}(w|K(z)), \qquad z \in W,
  \end{align*}
  as claimed by \eqref{form50}. The inequality $\mu(W) \gtrapprox A^{-s/3} \mu(B)$ follows from \eqref{ratio} and the definition of $B$. The bipartite condition holds with constants slightly worse than $t$. 

  Before continuing, we make a small further refinement of $W$. Cover $W$ by $\leq C(t/\epsilon)^{3}$ disjoint (dyadic) cubes $Q_{i}$ of side-length $\epsilon$. At most half of $\mu(W)$ can be contained in the union of those cubes $Q_{i}$ with $\mu(W \cap Q_{i}) \leq (\epsilon/t)^{3}\mu(W)/(2C)$. We refine $W$ by discarding the part of $W$ covered by these low-density cubes. At least half of the $\mu$ measure of $W$ remains, and now all the points $z \in W$ have the following property: they are contained in a cube $Q_{i} = Q_{i}(z)$ of side-length $\epsilon$ such that
  \begin{equation}\label{form69}
    \mu(W \cap Q_{i}) \gtrsim \left(\frac{\epsilon}{t}\right)^{3}\mu(W) \gtrapprox A^{-s/3}\epsilon^{3},
  \end{equation}
  using \eqref{measureOfW}. 

  At this point we observe that $\lambda_{\epsilon}$, as defined in \eqref{form95}, is fairly large. Namely, if $w$ lies in the high-density set defined in \eqref{form66}, for some $z \in W$, then by \eqref{form63} and \eqref{form50},
  \begin{equation}\label{form94}
    A_{\epsilon}^{s}\lambda_{\epsilon}^{-2s}\epsilon^{s} \lessapprox \m_{\epsilon} \leq m_{C_{1}\epsilon}^{\mu}(w|B) \leq \mu(B) \lesssim \mathbf{C} t^{s},
  \end{equation}
 where $A_{\epsilon}$ is the parameter defined in \eqref{form95}. Rearranging this inequality gives
 \begin{displaymath}
    \lambda_{\epsilon} \gtrapprox_{\mathbf{C},s} A_{\epsilon}^{1/2}\sqrt{\frac{\epsilon}{t}}.
  \end{displaymath}
  Recalling that $A_{\epsilon} \geq (A\epsilon)/(\nu \delta)$, then $\nu \leq \epsilon/\delta$, and finally $A \geq C_{\eta,\mathbf{C},s} \cdot \delta^{-\eta}$, we infer that 
  \begin{equation}\label{rhoDef}
    \rho := \floor{c \cdot \tfrac{\lambda_{\epsilon}}{\sqrt{\epsilon/t}}} \gtrapprox_{\mathbf{C},s} A_{\epsilon}^{1/2} \geq \left(C^{1/2}_{\eta,\mathbf{C},s} \right)\delta^{-\eta/2},
  \end{equation}
  where $c > 0$ is a small absolute constant to be specified momentarily. This shows in particular that $\rho \geq 1$, if $C_{\eta,\mathbf{C},s} \geq 1$ is large enough, depending here on $c,\mathbf{C}$ and $s$. 
  Thus, for a fixed point $z \in W$, \eqref{form66} and \eqref{form50} imply that it takes $\gtrsim \rho/c$ (in particular at least $\rho$) sets $I_{i}^{\epsilon}(z)$ to cover the high density set
  \begin{displaymath}
    H^{\epsilon}_{b}(z) := S^{\epsilon}(z) \cap \{w : m^{\mu}_{C_{1}\epsilon}(w|B \cap K(z)) \geq \m_{\epsilon} \text{ and } m^{\mu}_{N\epsilon}(w) \leq M_{\epsilon}\}.
  \end{displaymath}
  For a fixed point $z \in W$, we may hence choose $\rho$ points $v_{1},\ldots,v_{\rho} \in S^{\epsilon}(z)$, which are separated by a distance at least $C\sqrt{\epsilon/t}$ (here $C$ is another absolute constant, which may be chosen larger by making "$c$" smaller), and which satisfy
  \begin{equation}\label{form67}
    m^{\mu}_{C_{1}\epsilon}(v_{j}|K(z) \cap B) \geq \m_{\epsilon} \quad \text{and} \quad m^{\mu}_{N\epsilon}(v_{j}|W) \leq M_{\epsilon}, \qquad 1 \leq j \leq \rho.
  \end{equation}
  Fix $1 \leq j \leq \rho$, and consider the first condition in \eqref{form67}, which is shorthand for
  \begin{equation}\label{form68}
    \mu(B_{j}(z)) := \mu(\{z' \in B \cap K(z) : v_{j} \in S^{\epsilon}(z) \cap S^{C_{1}\epsilon}(z')\}) \geq \m_{\epsilon}.
  \end{equation}
  Whenever $z' \in B_{j}(z)$, then the circles $S(z)$ and $S(z')$ are $2\epsilon$-incident (since $z' \in K(z) = K_{\epsilon,t}(z)$ implies $\Delta'(z - z') \leq 2\epsilon$), and they are \textbf{both} $(N/2)$-tangent to a certain $(\epsilon,t)$-rectangle $R_{j}(z)$ containing $v_{j}$, for $N \geq 1$ large enough (depending on $C_{1}$, an absolute constant chosen at \eqref{form85}). Moreover, when $j \in \{1,\ldots,\rho\}$ varies, the corresponding $(\epsilon,t)$-rectangles are incomparable by the separation of the points $v_{j}$. We summarise the findings above: every $z \in W$ gives rise to $\rho$ incomparable $(\epsilon,t)$-rectangles $R_{j}(z)$, each being $(N/2)$-tangent to $S(z)$, having type $\geq \m_{\epsilon}$ with respect to the set $B$ according to \eqref{form68} (the notion of type was introduced in Definition \ref{def:type}), and containing a point $v_{j} = v_{j}(z)$. 
  
  To make the following discussion more rigorous, choose a maximal (finite) collection $\mathcal{R}$ of incomparable $(\epsilon,t)$-rectangles in $B(0,100)$. Then, by adjusting the constants appropriately, we may assume that each rectangle $R_{j}(z)$, as above, lies in $\mathcal{R}$.

  At this point, we also run one final pigeonholing argument. For $z \in W$ and $v_{j} = v_{j}(z)$ as above, we have the upper bound $m_{N\epsilon}^{\mu}(v_{j}(z)|W) \leq M_{\epsilon}$ by \eqref{form67}. This implies that
  \begin{equation}\label{form105}
    \mu(\{z' \in W : S(z') \text{ is $N$-tangent to } R_{j}(z)\}) \leq M_{\epsilon},
  \end{equation}
  because any circle $S(z')$ being $N$-tangent to $R_{j}(z)$ satisfies $v_{j}(z) \in R_{j}(z) \subset S^{N\epsilon}(z')$ by definition of $N$-tangency. On the other hand, $S(z)$ is $(N/2)$-tangent to $R_{j}(z)$ by the discussion above, and every circle $S(z')$ with $z' \in Q_{i}(z)$ (see above \eqref{form69}) is $N$-tangent to $R_{j}(z)$, hence
  \begin{displaymath}
    \mu(\{z' \in W : S(z') \text{ is $N$-tangent to } R_{j}(z)\}) \geq \mu(W \cap Q_{i}) \stackrel{\eqref{form69}}{\gtrapprox} A^{-s/3}\epsilon^{3} \geq \delta^{4},
  \end{displaymath}
  as we assumed at the start of the proof that $A \leq \delta^{-1}$. Now, for $z \in W$ fixed, we may pick a dyadic number $\delta^{4} \lessapprox \n_{\epsilon}(z) \leq M_{\epsilon}$ such that $\gtrapprox \rho$ rectangles $R_{j}(z)$ satisfy 
  \begin{displaymath}
    \n_{\epsilon}(z) \leq \mu(\{z' \in W : S(z') \text{ is $N$-tangent to } R_{j}(z)\}) \leq 2\n_{\epsilon}(z).
  \end{displaymath}
  Then, we may finally fix $\delta^{4} \lessapprox \n_{\epsilon} \leq M_{\epsilon}$, and a subset $W' \subset W$ with $\mu(W') \gtrapprox \mu(W)$, such that
  \begin{equation}\label{form90}
    \n_{\epsilon} \leq \mu(\{z' \in W : S(z') \text{ is $N$-tangent to } R_{j}(z)\}) \leq 2\n_{\epsilon}
  \end{equation}
  for $z \in W'$, and for $\gtrapprox \rho$ rectangles $R_{j}(z)$. From now on, the rectangles $R_{j}(z)$, $1 \leq j \leq \rho$, satisfying \eqref{form90} will be called the \emph{children} of $z \in W'$. According to \eqref{form90}, every child $R_{j}(z)$ of $z \in W'$ has type $\geq \n_{\epsilon}$ with respect to $W$, assuming that the notion of "tangency" has been defined as $N$-tangency; this is legitimate, since $N$ is an absolute constant.

  Every point $z \in W' \subset W$ gives rise to $\gtrapprox \rho$ children $R_{j}(z)$, as we just argued. Now, as $z \in W'$ varies, how many children in $\calR$ do we find in total, at least? If ten parents have three children each, and each child has at most two parents, then there are at least $3 \cdot 10/2 = 15$ children in total. For a more general statement, see Lemma \ref{fubiniLemma} below. Now, we do the same computation with "parents" replaced by points $z \in W'$ (children are, of course, the rectangles as before). We already know that every parent $z \in W'$ has $\gtrapprox \rho$ children in $\calR$, so we only need to find an upper bound for the number of parents.

  Fix a child $R = R_{j}(z)$, for some $z \in W'$, satisfying \eqref{form90}. If $z' \in W'$ is another parent with the same child $R$, then $S(z')$ is $(N/2)$-tangent to $R$ by definition, and in particular $N$-tangent to $R$. Thus, by \eqref{form90},
  \begin{displaymath}
    \mu(\{z' \in W' : R_{j}(z) \text{ is a child of } z'\}) \leq 2\n_{\epsilon}.
  \end{displaymath}
  Now, Lemma \ref{fubiniLemma} implies (take $(\Omega_{1},\mu_{1}) = (W',\mu)$, $\Omega_{2}$ the set of all possible rectangles $R_{j}(z) \in \calR$ with $z \in W$ and $1 \leq j \leq \rho$, $\mu_{2}$ the counting measure on $\Omega_{2}$, and $E = \{(z',R) \in \Omega_{1} \times \Omega_{2} : R \text{ is the child of } z'\}$) that the total number of rectangles $R \in \calR$, which are the child of some point $z \in W'$, is at least
  \begin{equation}\label{form72}
    \gtrapprox \frac{\mu(W') \rho}{\n_{\epsilon}} \gtrapprox \frac{\mu(W) \rho}{\n_{\epsilon}}.
  \end{equation}
  Moreover, every such child $R$ has type $(\geq \n_{\epsilon},\geq \m_{\epsilon})$ with respect to the $t$-bipartite set $W \cup B$ by \eqref{form68} and \eqref{form90} (as we already mentioned above, we define the concept of \emph{type}, recall Definition \ref{def:type}, using $N$-tangency). On the other hand, by Lemma \ref{wolffMeasures}, given $\tau > 0$, the maximal cardinality of incomparable $(\epsilon,t)$-rectangles of type $(\geq \n_{\epsilon},\geq \m_{\epsilon})$ is bounded from above by
  \begin{align*}
    \lesssim_{\tau} (\m_{\epsilon}\n_{\epsilon}\epsilon)^{-\tau} &\left( \left(\frac{\mu(W)\mu(B)}{\m_{\epsilon}\n_{\epsilon}} \right)^{3/4} + \frac{\mu(W)}{\n_{\epsilon}} + \frac{\mu(B)}{\m_{\epsilon}} \right)\\
    &\stackrel{\eqref{form92}}{\lessapprox} (\m_{\epsilon}\n_{\epsilon}\epsilon)^{-\tau} \left( \left(\frac{A^{s/3}\mu(W)^{2}}{\m_{\epsilon}\n_{\epsilon}} \right)^{3/4} + \frac{\mu(W)}{\n_{\epsilon}} + \frac{A^{s/3}\mu(W)}{\m_{\epsilon}} \right),
  \end{align*}
  recalling from \eqref{form92} that $\mu(B) \lessapprox A^{s/3}\mu(W)$. One can verify from \eqref{form94} that the hypothesis $\epsilon \leq c t$ in Lemma \ref{wolffMeasures} is satisfied if the lower bound $C_{\eta,\mathbf{C},s}$ for the constant $A$, and the constant "$C$" in the inequality $A_{\epsilon} \geq \log^{C}(1/\delta)A$ (see \eqref{form95}), are chosen large enough.

Now fix $0 < \tau < \eta s/50$. Since $s \in (0,1]$, $\delta^{4} \lesssim \n_{\epsilon} \leq M_{\epsilon} \lessapprox \m_{\epsilon}$ by \eqref{form91} and $\rho \gtrapprox A^{1/2} \geq \delta^{-\eta/2}$ by \eqref{rhoDef}, neither of the two latter terms can dominate \eqref{form72}. But the the first term cannot dominate either, since otherwise (importing the lower estimate for $\m_{\epsilon}$ from \eqref{form63}, recalling that $A_{\epsilon} \geq A$, and recalling the definition of $\rho$ from \eqref{rhoDef}),
  \begin{align*}
    \rho &\lessapprox \delta^{-10\tau}A^{s/4}\mu(W)^{1/2} \frac{\n_{\epsilon}^{1/4}}{\m_{\epsilon}^{3/4}}\\
    &\lessapprox \delta^{-10\tau}A^{s/4} t^{s/2} \m_{\epsilon}^{-1/2}\\
    &\lessapprox \delta^{-10\tau} A^{s/4} A_{\epsilon}^{-s/2} \left(\frac{t}{\epsilon}\right)^{s/2} \lambda^{s}_{\epsilon} = \delta^{-10\tau} A^{-s/4} \rho^{s}.
  \end{align*}
  This gives a contradiction, since $s \in (0,1]$, $\rho \geq 1$, and $A^{-s/4} \leq \delta^{\eta s/4}$. The proof of Lemma \ref{incidenceCircles} is complete.
\end{proof}

To finish this section, we verify the lemma used in the previous proof.

\begin{lemma}\label{fubiniLemma}
  Let $(\Omega_{1},\mu_{1}),(\Omega_{2},\mu_{2})$ be finite measure spaces, let $E \subset \Omega_{1} \times \Omega_{2}$ be a subset, and let $\pi_{1} \colon \Omega_{1} \times \Omega_{2} \to \Omega_{1}$ and $\pi_{2} \colon \Omega_{1} \times \Omega_{2} \to \Omega_{2}$ be the coordinate projections. If $E \subset \Omega_{1} \times \Omega_{2}$ is $\mu_{1} \times \mu_{2}$ measurable,
  \begin{displaymath}
    \mu_{2}(\{\omega_{2} \in \Omega_{2} : (\omega_{1},\omega_{2}) \in E\}) \geq C_{2}
  \end{displaymath}
  for all $\omega_{1} \in \pi_{1}(E)$, and
  \begin{displaymath}
    \mu_{1}(\{\omega_{1} \in \Omega_{1} : (\omega_{1},\omega_{2}) \in E\} \leq C_{1}
  \end{displaymath}
  for all $\omega_{2} \in \pi_{2}(E)$, then
  \begin{displaymath}
    \mu_{2}(\pi_{2}(E)) \geq \frac{C_{2}}{C_{1}} \mu_{1}(\pi_{1}(E)).
  \end{displaymath}
\end{lemma}

\begin{proof}
  This is an easy application of Fubini's theorem:
  \begin{align*}
    C_{2} \mu_{1}(\pi_{1}(E)) &\leq \int_{\pi_{1}(E)} \mu_{2}(\{\omega_{2} \in \Omega_{2} : (\omega_{1},\omega_{2}) \in E\}) \dd\mu_{1} \omega_{1}\\
    &= (\mu_{1} \times \mu_{2})(E) = \int_{\pi_{2}(E)} \mu_{1}(\{\omega_{1} \in \Omega_{1} : (\omega_{1},\omega_{2}) \in E\}) \dd\mu_{2} \omega_{2}\\
    &\leq C_{1} \mu_{2}(\pi_{2}(E)),
  \end{align*} 
  which gives the claim by rearranging.
\end{proof}

\section{A measure-theoretic variant of Schlag's lemma for sine waves}

In this section, we prove a variant of Lemma \ref{incidenceCircles} for the sine waves 
\begin{displaymath}
  \Gamma(z) = \{(\theta,\gamma(\theta) \cdot z) : \theta \in J/2\},
\end{displaymath}
where 
\begin{displaymath}
  \gamma(\theta) = \tfrac{1}{\sqrt{2}}(\cos \theta,\sin \theta,1),
\end{displaymath}
and $J \subset [0,2\pi)$ is a short compact interval with $2J \subset [0,2\pi)$. We assume that $J$ is so short that Lemma \ref{Edelta} applies, and so does the discussion in Section \ref{tangencySection}. In accordance with earlier notation, we write
\begin{displaymath}
  \Delta(z) := \Delta_{J}(z) = \min_{\theta \in J} |\pi_{V_{\theta}}(z)|.
\end{displaymath}
Recall that 
\begin{displaymath}
  E_{\delta}(z) = \{\theta \in J/2 : |\rho_{\theta}(z)| \leq \delta\} \quad \text{and} \quad \Gamma^{\delta}(z) = \{(\theta,\theta') \in \tfrac{J}{2} \times \R : |\theta' - \rho_{\theta}(z)| \leq \delta\},
\end{displaymath}
where $\rho_{\theta}(z) = \gamma(\theta) \cdot z$. Given a finite measure $\mu$ on $\R^{3}$ and $\delta > 0$, we re-define the multiplicity function $m_{\delta}^{\mu} \colon \R^{2} \to [0,\mu(\R^{3})]$ in the obvious way:
\begin{displaymath}
  m^{\mu}_{\delta}(w) = \mu(\{z' \in \R^{3} : w \in \Gamma^{\delta}(z')\}).
\end{displaymath}
With this notation, we have the following perfect analogue of Lemma \ref{incidenceCircles} (the only change is literally that $S$ is replaced by $\Gamma$):

\begin{lemma}\label{incidenceSineCurves}
  Fix $s \in (0,1]$, $\delta > 0$, $\eta > 0$, $\mathbf{C} \geq 1$, and $A \geq C_{\eta,\mathbf{C},s} \cdot \delta^{-\eta}$, where $C_{\eta,\mathbf{C},s} \geq 1$ is a large constant depending only on $\eta,\mathbf{C}$, and $s$. Let $\mu$ be a probability measure on $\R^{3}$ satisfying the Frostman condition $\mu(B(z,r)) \leq \mathbf{C} r^{s}$ for all $z \in \R^{3}$ and $r > 0$, and with $K := \spt \, \mu \subset \mathbf{B}_{0}$.  Then, for $\lambda \in (0,1]$, there is a set $G(A,\delta,\lambda) \subset K$ with 
  \begin{displaymath}
    \mu(K \setminus G(A,\delta,\lambda)) \leq A^{-s/3}
  \end{displaymath}
  such that the following holds for all $z \in G(A,\delta,\lambda)$:
  \begin{displaymath}
    |\Gamma^{\delta}(z) \cap \{w : m^{\mu}_{\delta}(w) \geq A^{s}\lambda^{-2s} \delta^{s}\}| \leq \lambda |\Gamma^{\delta}(z)|.
  \end{displaymath}
\end{lemma}

\begin{remark}
  We will assume that the reader is already familiar with the proof of Lemma \ref{incidenceCircles} above; if so, we can promise that \ref{incidenceSineCurves} is easy reading, as the structure of the argument is exactly the same. Even at the risk of repetition, we will still include most details. Apart from a few notational changes, the main difference occurs at the end of the proof. In the previous argument, we were counting tangent circles in two different ways. Below, the natural analogue would be to count tangent sine waves, but we do not have a "sine wave variant" of Wolff's incidence bound, Lemma \ref{wolffMeasures}, at our disposal. So, instead, we use the discussion in Section \ref{tangencySection} to infer that "many tangent sine waves imply many tangent circles", and then we can literally apply Lemma \ref{wolffMeasures} again. Finally, we also need to apply Lemma \ref{incidenceCircles} on the last few meters of the proof: information from the lemma will replace the appeal to the "induction hypothesis" within Lemma \ref{incidenceCircles} (for somewhat complicated technical reasons, the corresponding induction hypothesis appears to be too weak to settle the proof in the setting below).
\end{remark}

\begin{proof}[Proof of Lemma \ref{incidenceSineCurves}] Just like in the proof of Lemma \ref{incidenceCircles}, we may assume that $\delta > 0$ is small in a manner depending on $\eta,\mathbf{C},s$, in particular $\delta \in (0,\tfrac{1}{2}]$, and $A \leq \delta^{-1}$.

  Assume to the contrary that there exists a dyadic number $\delta \in 2^{-\N}$, and a number 
  \begin{equation}\label{form51-sine}
    m \geq A^{s}\lambda^{-2s}\delta^{s}
  \end{equation}
  such that 
  \begin{equation*}
    |\Gamma^{\delta}(z) \cap \{w : m_{\delta}^{\mu}(w) \geq m\}| \geq \lambda |\Gamma^{\delta}(z)|
  \end{equation*}
  for every $z \in D \subset K$, where
  \begin{equation*}
    \mu(D) > A^{-s/3}.
  \end{equation*}
  This will result in a contradiction provided that the constant $C_{\eta,\mathbf{C},s}$ in the assumption $A \geq C_{\eta,\mathbf{C},s} \cdot \delta^{-\eta}$ is sufficiently large. For $z \in \R^{3}$ and dyadic numbers $\epsilon,t \in (0,1]$, define
  \begin{equation*}
    K_{\epsilon,t}(z) := \{z' \in K \colon \Gamma^{\delta}(z) \cap \Gamma^{\delta}(z') \neq \emptyset, \: t \leq |z - z'| < 2t \text{ and }  \epsilon \leq \Delta(z - z') \leq 2\epsilon\}.
  \end{equation*}
  In the case $\epsilon = \delta$, we again drop the lower constraint from $\Delta(z - z')$ (as in the proof of Lemma \ref{incidenceCircles}). Define also the restricted multiplicity function
  \begin{displaymath}
    m^{\mu}_{\delta}(w | K_{\epsilon,t}(z)) := \mu(\{z' \in K_{\epsilon,t}(z) : w \in \Gamma^{\delta}(z')\}).
  \end{displaymath}
  Proceeding as in the proof of Lemma \ref{incidenceCircles}, if the constant $C_{\eta,\mathbf{C},s} \geq 1$ is taken large enough, we may pigeonhole fixed dyadic numbers $\epsilon \in [\delta,1]$ and $t \in [2C\delta,1]$ (with $C = C(\gamma,J) \geq 1$ now explicitly being the constant from Lemma \ref{Edelta}), $m \lessapprox \bar{m} \leq m$ and $\lambda \lessapprox \bar{\lambda} \leq \lambda$, and a subset $\bar{D} \subset D$ with $\mu(\bar{D}) \gtrapprox \mu(D)$, such that the following holds for all $z \in \bar{D}$:
  \begin{equation}\label{barD}
    |H^{\delta}(z)| := |\Gamma^{\delta}(z) \cap \{w :  m^{\mu}_{\delta}(w | K_{\epsilon,t}(z)) \geq \bar{m}\}| \geq \bar{\lambda}|\Gamma^{\delta}(z)|.
  \end{equation}
  For the rest of the proof, the numbers $t$ and $\epsilon$ will be fixed, and we write $K_{\epsilon,t}(z) =: K(z)$.

  For a heuristic explanation of what happens next, see the corresponding spot in the proof of Lemma \ref{incidenceCircles}. Fix $z \in \bar{D}$, so that \eqref{barD} holds. We claim that there exists a dyadic number $\nu = \nu(z) \in \{1,\ldots,\epsilon/\delta\}$, and an absolute constant $C_{1} \geq 1$, such that
  \begin{equation}\label{form62a}
    |\Gamma^{\epsilon}(z) \cap \{w : m^{\mu}_{C_{1}\epsilon}(w|K(z)) \gtrsim \nu\bar{m}\}| \gtrapprox \frac{\bar{\lambda}\epsilon}{\nu \delta} |\Gamma^{\epsilon}(z)|.
  \end{equation}
  To see this, we recall from Lemma \ref{Edelta} that if $z,z' \in \R^{3}$ and with $|z - z'| \geq t \geq C\delta$ and $\Delta(z - z') \sim \epsilon$, then $\Gamma^{\delta}(z) \cap \Gamma^{\delta}(z')$ can be covered by two vertical tubes of width $\leq C\delta/\sqrt{\epsilon t}$ (this remains true if $\epsilon = \delta$, and merely $\Delta(z - z') \leq 2\epsilon$). Motivated by this, we first divide $J/2$ into \emph{short intervals} $J_{1},\ldots,J_{N}$ of length $C\delta/\sqrt{\epsilon t}$. Consider the corresponding \emph{thin tubes} $T_{j} = J_{j} \times \R$. Since $|T_{j} \cap \Gamma^{\delta}(z)| \leq \delta^{2}/\sqrt{\epsilon t}$, we may, by \eqref{barD}, find at least a constant times
  \begin{displaymath}
    \frac{\bar{\lambda}|\Gamma^{\delta}(z)|}{\delta^{2}/\sqrt{\epsilon t}} \sim \frac{\bar{\lambda}\sqrt{\epsilon t}}{\delta}
  \end{displaymath}
  indices $j$ such that $H^{\delta}(z) \cap T_{j} \neq \emptyset$. Denote these indices by $\calJ(z)$, and for each $j \in \calJ(z)$, pick a point $w_{j} \in H^{\delta}(z) \cap T_{j}$. Thus $m^{\mu}_{\delta}(w_{j}|K(z)) \geq \bar{m}$ for $j \in \calJ$. Throw away at most half of the indices to ensure that $|w_{i} - w_{j}| \geq C\delta/\sqrt{\epsilon t}$ for $i,j \in \calJ(z)$ with $i \neq j$. Then, the sets
  \begin{displaymath}
    K_{j}(z) := \{z' \in K(z) : w_{j} \in \Gamma^{\delta}(z) \cap \Gamma^{\delta}(z')\}
  \end{displaymath}
  with $m_{\delta}^{\mu}(w_{j}|K(z)) = \mu(K_{j}(z))$ have bounded overlap:
  \begin{equation}\label{form97}
    \sum_{j \in \calJ(z)} \chi_{K_{j}(z)}(z') \leq 2, \qquad z' \in K(z).
  \end{equation}
  Indeed, if $z' \in K_{j}(z)$, then $w_{j} \in \Gamma^{\delta}(z) \cap \Gamma^{\delta}(z')$, which implies that $w_{j}$ has to lie in one of the at most two vertical tubes of width at most $\delta/\sqrt{\epsilon t}$ covering the intersection $\Gamma^{\delta}(z) \cap \Gamma^{\delta}(z')$. By the separation of the points $w_{j}$, this can happen for at most two values of $j$.

  Next, we group the points $w_{j}$ inside somewhat thicker vertical tubes. To this end, divide $J/2$ into \emph{long intervals} $I_{1},\ldots,I_{M}$ of length $C\sqrt{\epsilon/t}$. By adjusting the lengths appropriately, we may assume that the long intervals $I_{i}$ are sub-divided further into
  \begin{displaymath}
    \frac{|I_{i}|}{|J_{j}|} \leq \frac{\sqrt{\epsilon/t}}{\delta/\sqrt{\epsilon t}} = \frac{\epsilon}{\delta}
  \end{displaymath}
  short intervals $J_{j}$. For each interval $I_{i}$, write
  \begin{displaymath}
    k(i) := \card \{j \in \calJ : w_{j} \in \mathbf{T}_{i} \cap \Gamma^{\delta}(z)\},
  \end{displaymath} 
  where $\mathbf{T}_{i}$ is the \emph{thick tube} $\mathbf{T}_{i} := I_{i} \times \R$. Since $0 \leq k(i) \leq \epsilon/\delta$, there is a dyadic number $\nu = \nu(z) \in \{1,\ldots,\epsilon/\delta\}$ such that $\gtrapprox |\calJ| \gtrsim \bar{\lambda}\sqrt{\epsilon t}/\delta$ points $w_{j}$ are contained in the union of the thick tubes $\mathbf{T}_{i}$ with $\nu \leq k(i) \leq 2\nu$. Denote the indices of these thick tubes by $\calI(z)$. Thus, if $i \in \calI(z)$, then $\mathbf{T}_{i} \cap \Gamma^{\delta}(z)$ contains at least $\nu$ points $w_{j}$, and
  \begin{equation}\label{form98}
    |\calI(z)| \gtrapprox \frac{|\calJ|}{\nu} \gtrsim \frac{\bar{\lambda} \sqrt{\epsilon t}}{\nu \delta}.
  \end{equation}

  Fix $i \in \calI(z)$ and $w_{j} \in \mathbf{T}_{i} \cap \Gamma^{\delta}(z)$. We claim that whenever $z' \in K_{j}(z)$, then 
  \begin{equation}\label{form85a}
    \mathbf{T}_{i} \cap \Gamma^{\epsilon}(z) \subset \Gamma^{C_{1}\epsilon}(z'),
  \end{equation}
  for some $C_{1} \geq 1$ large enough. To see this, note that by definition of $z' \in K_{j}(z)$, we have 
  \begin{displaymath}
    (w_{j}^{1},w_{j}^{2}) := w_{j} \in \mathbf{T}_{i} \cap \Gamma^{\delta}(z) \cap \Gamma^{\delta}(z').
  \end{displaymath}
  Thus
  \begin{displaymath}
    |\rho_{w_{j}^{1}}(z) - \rho_{w_{j}^{1}}(z')| \leq 2\delta,
  \end{displaymath}
  or, in other words, $w_{j}^{1} \in E_{2\delta}(z - z')$. Since $\Delta(z - z') = \Delta_{J}(z - z') \leq \epsilon$ and $|z - z'| \geq t \geq C(2\delta)$, Lemma \ref{Edelta} says that $w_{j}^{1}$ is at distance at most a constant times $\sqrt{\epsilon/t}$ from a certain point $\theta_{0} \in 2J$ with the properties that 
  \begin{equation}\label{form96}
    \dot{\gamma}(\theta_{0}) \cdot (z - z') = 0 \quad \text{and} \quad |\gamma(\theta_{0}) \cdot (z - z')| \lesssim \epsilon.
  \end{equation}
  Now, we can prove \eqref{form85a}: fix a point $w = (w^{1},w^{2}) \in \mathbf{T}_{i} \cap \Gamma^{\epsilon}(z)$, and note that $|w^{1} - \theta_{0}| \leq |w^{1} - w_{j}^{1}| + |w^{1}_{j} - \theta_{0}| \lesssim \sqrt{\epsilon/t}$, and $|\rho_{w^{1}}(z) - w^{2}| \leq \epsilon$ by definition of $w \in \mathbf{T}_{i} \cap \Gamma^{\epsilon}(z)$. It follows, using \eqref{form96}, that
  \begin{align*}
    |\rho_{w^{1}}(z') - w^{2}| &\leq |\gamma(w^{1}) \cdot (z' - z)| + |\rho_{w^{1}}(z) - w^{2}|\\
    &\leq \int_{\theta_{0}}^{w^{1}} |\dot{\gamma}(s) \cdot (z' - z)| \dd s + |\gamma(\theta_{0}) \cdot (z' - z)| + \epsilon\\
    &\lesssim \int_{\theta_{0}}^{w^{1}} \int_{\theta_{0}}^{s} |\ddot{\gamma}(r) \cdot (z' - z)| \dd r \dd s + \epsilon\\
    &\lesssim |z' - z| |w^{1} - \theta_{0}|^{2} + \epsilon \lesssim \epsilon.
  \end{align*} 
  This is another way of writing $w \in \Gamma^{C_{1}\epsilon}(z')$, so the proof of \eqref{form85a} is complete.

  Now, for $i \in \calI(z)$ and $w \in \mathbf{T}_{i} \cap \Gamma^{\epsilon}(z)$ fixed, we can use the bounded overlap of the sets $K_{j}(z)$ (recall \eqref{form97}) and \eqref{form85a} to obtain
  \begin{displaymath}
    m_{C_{1}\epsilon}^{\mu}(w|K(z)) = \int_{K(z)} \chi_{\Gamma^{C_{1}\epsilon}(z')}(w) \dd\mu z' \gtrsim \sum_{w_{j} \in \mathbf{T}_{i} \cap \Gamma^{\delta}(z)} \mu(K_{j}(z)) \geq \nu \bar{m}.
  \end{displaymath}
  (See the corresponding spot in the proof of Lemma \ref{incidenceCircles}, namely the calculations above \eqref{form108}, for more details.) Thus, we have proven that whenever $i \in \calI(z)$, then $\mathbf{T}_{i} \cap \Gamma^{\epsilon}(z) \subset \Gamma^{\epsilon}(z) \cap \{w : m_{C_{1}\epsilon}^{\mu}(w|K(z)) \gtrsim \nu \bar{m}\}$. Recalling \eqref{form98}, this proves that
  \begin{displaymath}
    |\Gamma^{\epsilon}(z) \cap \{w : m_{C_{1}\epsilon}^{\mu}(w|K(z)) \gtrsim \nu \bar{m}\}| \geq |\calI(z)| |\mathbf{T}_{i} \cap \Gamma^{\epsilon}(z)| \gtrapprox \frac{\bar{\lambda}\sqrt{\epsilon t}}{\nu \delta} \sqrt{\epsilon/t} \epsilon \sim \frac{\bar{\lambda}\epsilon}{\nu \delta} |\Gamma^{\epsilon}(z)|,
  \end{displaymath}
  which is precisely \eqref{form62a}. 

  Recall that the dyadic number $\nu = \nu(z)$ still depends on the point $z \in \bar{D}$. We pass to a subset of measure $\gtrapprox \mu(\bar{D}) \gtrapprox A^{-s/3}$ to make the choice uniform. With this reduction, we may assume that \eqref{form62a} holds for all $z \in \bar{D}$, for some fixed dyadic number $\nu \in \{1,\ldots,\epsilon/\delta\}$. 

  We now re-write \eqref{form62a} slightly, by denoting
  \begin{equation}\label{form110}
    \lambda_{\epsilon} := \log^{-C}(1/\delta) \frac{\bar{\lambda}\epsilon}{\nu \delta} \quad \text{and} \quad A_{\epsilon} := \log^{C}(1/\delta) \left(\frac{A \epsilon}{\nu \delta} \right),
  \end{equation}
  where $C \geq 1$ is a suitable constant. Recall from \eqref{form51-sine} that $\bar{m} \approx m \geq A^{s}\lambda^{-2s}\delta^{s} \approx A^{s}\bar{\lambda}^{-2s}\delta^{s}$. Since also $s \in (0,1]$, and $\nu \geq 1$, we have 
  \begin{equation*}
    \m_{\epsilon} := \frac{\nu \bar{m}}{C} \gtrapprox \nu A^{s} \bar{\lambda}^{-2s}\delta^{s} \gtrapprox \left(\frac{A \epsilon}{\nu \delta} \right)^{s} \lambda_{\epsilon}^{-2s} \epsilon^{s} \approx A_{\epsilon}^{s}\lambda^{-2s}_{\epsilon}\epsilon^{s}.
  \end{equation*}
  Thus, if $C \geq 1$ is large enough, \eqref{form62a} implies that
  \begin{equation}\label{form64a}
    |H^{\epsilon}(z)| := |\Gamma^{\epsilon}(z) \cap \{w : m^{\mu}_{C_{1}\epsilon}(w|K(z)) \geq \m_{\epsilon}\}| \geq \frac{10\lambda_{\epsilon}}{|J|} |\Gamma^{\epsilon}(z)|, \qquad z \in \bar{D}.
  \end{equation}
  We will denote the first coordinates of $H^{\epsilon}(z)$ by $H_{1}^{\epsilon}(z) := \{w^{1} \in \tfrac{J}{2} : (w^{1},w^{2}) \in H^{\epsilon}(z)\}$. We record that \eqref{form64a} implies
  \begin{equation}\label{form99}
    |H_{1}^{\epsilon}(z)| \geq 2\lambda_{\epsilon}.
  \end{equation}
  Otherwise, by Fubini, $|H^{\epsilon}(z)| \leq 2\lambda_{\epsilon} \cdot 2\epsilon < (10\lambda_{\epsilon}/|J|) |\Gamma^{\epsilon}(z)|$.
  
   Next, we extract two sets
  \begin{displaymath}
    W \subset \bar{D} \quad \text{and} \quad B \subset \R^{3}
  \end{displaymath}
  satisfying the two $t$-bipartite conditions $\max\{\diam(W),\diam(B)\} \leq t \leq \dist(W,B)$ and $\diam(W \cup B) \leq 100t$. We will moreover find $W$ and $B$ so that 
  \begin{equation*}
    \mu(W) \gtrapprox A^{-s/3} \cdot \mu(B),
  \end{equation*}
  and
  \begin{equation*}
    m^{\mu}_{C_{1}\epsilon}(w |K(z)) = m^{\mu}_{C_{1}\epsilon}(w |K(z) \cap B), \qquad z \in W, \: w \in \R^{2},
  \end{equation*}
  and every $z \in W$ is contained in a dyadic cube $Q(z)$ of side-length $\epsilon$ and mass
  \begin{equation*}
    \mu(W \cap Q(z)) \gtrapprox A^{-s/3}\epsilon^{3}.
  \end{equation*}
  The sets $W,B$ are found by verbatim the same argument as in the proof of Lemma \ref{incidenceCircles}, so we omit the details. 

  At this point, the proof deviates from its analogue for circles. We apply the variant of the current lemma for circles -- namely Lemma \ref{incidenceCircles} -- to the collection of circles $S(z) = S(x,r)$ with $z \in K$.\footnote{This is difficult to explain heuristically at the moment, but we make the following attempt. The plan is eventually pass from "sine waves with high multiplicity" to "circles with plenty of tangencies", using Lemma \ref{sineCurvesToCircles}. But we will also need to know that there are not \textbf{too} many tangencies between the circles. It seems that having (upper) multiplicity control for the sine waves is a bit too weak to get that, and so we, instead, secure multiplicity control for the circles directly. Such control is provided by Lemma \ref{incidenceCircles}.} For this purpose, we define the circular multiplicity function
  \begin{displaymath}
    m^{\mu,S}_{\epsilon}(w) := \mu(\{z' : w \in S^{\epsilon}(z')\}).
  \end{displaymath}
  (Recall that $K \subset \mathbf{B}_{0}$ lies in the upper half-space, so every point $z = (x,r)$ with $z \in \spt \mu$ corresponds to an honest circle $S(x,r)$.) Then we apply Lemma \ref{incidenceCircles} at scale $C_{2}\epsilon$ for a suitable $C_{2} \geq C_{1} \geq 1$ (to be determined later), and with the constants 
  \begin{displaymath}
    A_{\epsilon} \geq A \geq C_{\eta,\mathbf{C},s}\epsilon^{-\eta} \geq C_{\eta,\mathbf{C},s}(C_{2}\epsilon)^{-\eta}
  \end{displaymath}
  and $\lambda_{\epsilon}/(CC_{2}) > 0$ (here $C \geq 1$ is a less relevant constant, just large enough so that \eqref{form65a} below holds). The conclusion is that there exists a set $G = G(C_{2}\epsilon,\lambda_{\epsilon}/(CC_{2})) \subset K$ with $\mu(K \setminus G) \leq A_{\epsilon}^{-s/3}$ such that
  \begin{equation}\label{form65a}
    |S^{C_{2}\epsilon}(z) \cap \{w : m^{\mu,S}_{C_{2}\epsilon}(w) \geq A_{\epsilon}^s[\lambda_{\epsilon}/(CC_{2})]^{-2s}(C_2\epsilon)^s\}| \leq \frac{\lambda_{\epsilon}}{CC_{2}}|S^{C_{2}\epsilon}(z)| \leq \frac{\lambda_{\epsilon}}{10}|S^{\epsilon}(z)|
  \end{equation}
  for $z \in G$. In particular, since $A_{\epsilon} \geq \log^{C}(1/\delta) A$, and $\mu(\bar{D}) \gtrapprox A^{-s/3}$, the estimate \eqref{form65a} holds for at least half of the points $z \in \bar{D}$ (assuming that $C$ was chosen large enough). We restrict attention to this half, so that \eqref{form99} and \eqref{form65a} hold simultaneously for all $z \in \bar{D}$. 

  What we want to infer from \eqref{form65a} is the following: Fix $z = (x,r) \in \bar{D}$ and a point $w = (w^{1},w^{2}) \in H^{\epsilon}(z)$, so that $w^{1} \in H_{1}^{\epsilon}(z)$. Then, consider the ray $\ell_{x,w^{1}}$ emanating from $x$ and passing through $x + r(\cos w_{1},\sin w_{1})$. Assume that the intersection $S^{\epsilon}(z) \cap \ell_{x,w^{1}}$ is contained in the set on the left hand side of \eqref{form65a}. Now, if this happened for all $w^{1} \in H_{1}^{\epsilon}(z)$, then the set on the left hand side of \eqref{form65a} would evidently have measure at least $|H_{1}^{\epsilon}(z)| \epsilon > (\lambda_{\epsilon}/10)|S^{\epsilon}(z)|$, which is ruled out by \eqref{form65a}. In fact, by the same argument, there exists a subset $\tilde{H}^{\epsilon}_{1}(z) \subset H^{\epsilon}_{1}(z)$ of length 
  \begin{equation}\label{form100}
    |\tilde{H}^{\epsilon}_{1}(z)| \geq \lambda_{\epsilon}
  \end{equation}
  such that the following two things hold:
  \begin{itemize}
    \item[(a)] For every $w^{1} \in \tilde{H}_{1}^{\epsilon}(z)$, there exists $w^{2} \in \R$ such that $w = (w^{1},w^{2}) \in H^{\epsilon}(z)$.
    \item[(b)] For every $w^{1} \in \tilde{H}_{1}^{\epsilon}(z)$, the intersection $S^{\epsilon}(z) \cap \ell_{x,w^{1}}$ contains a point $v = v(w^{1},z)$ with $m^{\mu,S}_{C_{2}\epsilon}(v) \leq A_{\epsilon}^s[\lambda_{\epsilon}/(CC_{2})]^{-2s}(C_2\epsilon)^s =: M_{\epsilon}$.
  \end{itemize}

  Since the definitions of $\lambda_{\epsilon}$ and $A_{\epsilon}$ (as in \eqref{form110}) are the same as in the proof of Lemma \ref{incidenceCircles}, we may repeat the computations from around \eqref{rhoDef} to conclude that $\epsilon$ is significantly smaller than $t$, and
  \begin{equation*}
    \rho := \floor{c \cdot \frac{\lambda_{\epsilon}}{\sqrt{\epsilon/t}}} \gtrapprox_{\mathbf{C},s} (C_{\eta,\mathbf{C},s}^{1/2}) \cdot \delta^{-\eta/2}.
  \end{equation*}
  In particular, $\rho \geq 1$. Thus, for $z \in W$ fixed, it takes, by \eqref{form100}, at least $\rho$ long intervals $I_{1},\ldots,I_{\rho}$ (of length $C\sqrt{\epsilon/t}$ as before) to cover the set $\tilde{H}_{1}^{\epsilon}(z)$. We may in particular choose $\rho$ points $w_{1}^{1},\ldots,w_{\rho}^{1} \in \tilde{H}^{1}_{\epsilon}(z)$, which are separated by at least $C\sqrt{\epsilon/t}$, and which by (a) from the definition of $\tilde{H}^{1}_{\epsilon}(z)$ satisfy
  \begin{equation}\label{form67a}
    m^{\mu}_{C_{1}\epsilon}((w^{1}_{j},w^{2}_{j})|K(z) \cap B) \geq \m_{\epsilon}, \qquad 1 \leq j \leq \rho,
  \end{equation}
  for certain choices of $w^{2}_{j} \in \R$ such that $w_{j} := (w^{1}_{j},w^{2}_{j}) \in \Gamma^{\epsilon}(z)$. Unwrapping the definition, we re-write \eqref{form67a} as
  \begin{equation*}
    \mu(\{z' \in B \cap K(z) : w_{j} \in \Gamma^{C_{1}\epsilon}(z) \cap \Gamma^{C_{1}\epsilon}(z')\}) \geq \m_{\epsilon}.
  \end{equation*}
  Now, fix $z \in W$ and $z' \in B \cap K(z)$ with $w_{j} \in \Gamma^{C_{1}\epsilon}(z) \cap \Gamma^{C_{1}\epsilon}(z')$. If we write $z = (x,r)$ and $z' = (x',r')$, then by Lemma \ref{sineCurvesToCircles}, the circles $S(x,r)$ and $S(x',r')$ are both $C$-tangent to an $(\epsilon,t)$-rectangle $R_{j}(z)$ with
  \begin{equation}\label{form74}
    R_{j}(z) \subset S^{C\epsilon}(x,r) \cap B(x + r(\cos w_{j}^{1},\sin w_{j}^{1}),C_{2}\sqrt{\epsilon/t}),
  \end{equation}
  where $C \geq 1$ is a constant depending only on $C_{1}$ (which was an absolute constant). As $z \in W$ is fixed when $j \in \{1,\ldots,\rho\}$ varies, the rectangles $R_{j}(z)$ are incomparable by \eqref{form74}, and the separation of the points $w_{j}^{1}$. So, every $z \in W$ gives rise to $\rho$ incomparable $(\epsilon,t)$-rectangles, all of which are $C$-tangent to $S(z)$, and have type $\geq \m_{\epsilon}$ with respect to the set $B$. This is nearly a perfect analogue of the conclusion we drew after \eqref{form68} in the proof of Lemma \ref{incidenceCircles}, but one crucial feature is missing: the rectangles $R_{j}(z)$ do not (yet) contain suitable analogues of the points $v_{j}(z)$, for which there is also an \textbf{upper} bound for multiplicity, compare with \eqref{form67}. To remedy this, we need (b) from the definition of $\tilde{H}_{1}^{\epsilon}$: namely, for $z = (x,r) \in W$ fixed and $1 \leq j \leq \rho$, we may pick $v_{j}(z) \in S^{\epsilon}(z) \cap \ell_{x,w_{j}^{1}}$ satisfying
  \begin{equation}\label{form104}
    m_{C_{2}\epsilon}^{\mu,S}(v_{j}(z)) \leq M_{\epsilon}.
  \end{equation}
  Note that $v_{j}(z)$ lies close to $R_{j}(z)$ by \eqref{form74}, and the definition of $\ell_{x,w_{j}^{1}}$. In fact, if $S(x',r')$ is any circle tangent to the $(\epsilon,t)$-rectangle $R_{j}(z)$, then $S(x',r')$ is tangent (with slightly different constants) to any rectangle comparable to $R_{j}(z)$, and in particular to an $(\epsilon,t)$-rectangle $R'$ with $v_{j}(z) \in R \subset S^{\epsilon}(z)$. If $C_{2} \geq 1$ was chosen large enough, then this implies that $v_{j}(z) \in S^{C_{2}\epsilon}(x',r')$. Combined with \eqref{form104}, this shows that
  \begin{displaymath}
    \mu(\{z' : S(z') \text{ is tangent to } R_{j}(z)\}) \leq \mu(\{z' : v_{j}(z) \in S^{C_{2}\epsilon}(z')\}) = m_{C_{2}\epsilon}^{\mu,S}(v_{j}(z)) \leq M_{\epsilon}, \quad z \in W.
  \end{displaymath}
  This is an exact analogue of \eqref{form105}. 

  After this, the proof runs exactly in the same manner as that of Lemma \ref{incidenceCircles}. First, one finds by pigeonholing a number $n_{\epsilon}$ with $\delta^{4} \lessapprox \n_{\epsilon} \leq M_{\epsilon}$ such that 
  \begin{equation*}
    \n_{\epsilon} \leq \mu(\{z' \in W : S(z') \text{ is tangent to } R_{j}(z)\}) \leq 2\n_{\epsilon}
  \end{equation*}
  for all $z \in W'$ with $\mu(W') \gtrapprox \mu(W)$, and for $\gtrapprox \rho$ values of $j$. This is the analogue of \eqref{form90}, and the proof is the same. These rectangles $R_{j}(z)$ are then again called the \emph{children} of $z \in W'$, and one observes that they have type $(\geq \n_{\epsilon},\geq \m_{\epsilon})$ with respect to the $t$-bipartite set $W \cup B$. The same arguments as in the proof of Lemma \ref{incidenceCircles} now give upper and lower bounds for the family of all rectangles $R_{j}(z)$, arising from $z \in W'$ and $1 \leq j \leq \rho$; comparing these bounds against each other produces a contradiction as before, and completes the proof of Lemma \ref{incidenceSineCurves}.
\end{proof}

\section{Proof of the main result}

We are now ready to prove the main result, Theorem \ref{main2}, which we recall here.

\begin{thm}\label{mainExc}
  Let $K \subset \R^{3}$ be an analytic set with $0 < \Hd K \leq 1$, and let $0 \leq t < \Hd K$. Then $\Hd \rho_{\theta}(K) \geq t$ for all $\theta \in [0,2\pi) \setminus E$, where
  \begin{displaymath}
    \Hd E \leq \frac{\Hd K + t}{2\Hd K} < 1.
  \end{displaymath}
\end{thm} 

Note that in Theorem \ref{mainExc}, we can assume without loss of generality that $K \subset \mathbf{B}_{0} $, where $\mathbf{B}_{0}$ is defined in Definition \ref{B0}.  Indeed, for any $\epsilon > 0$, we may find $z_{\epsilon} \in \R^{3}$ such that $\Hd [(K + z_{\epsilon}) \cap \mathbf{B}_{0}] \geq \Hd K - \epsilon$. Then we just observe that $\Hd \rho_{\theta}(K) = \Hd \rho_{\theta}(K + z_{\epsilon})$ for all $\theta \in [0,2\pi)$, by the linearity of $\rho_{\theta}$. With this reduction in mind (and recalling \eqref{slicesEq}), Theorem \ref{mainExc} follows immediately from the next result:

\begin{thm}\label{mainint}
  Let $K \subset \mathbf{B}_{0}$ be an analytic set with $\Hd K \leq 1$ and let $0 \leq t < \Hd K$. Let $\calL_{t}$ be the set of all vertical lines $L_{\theta} := \{(\theta,y) : y \in \R\} \subset \R^{2}$, $\theta \in [0,2\pi)$, such that
  \begin{displaymath}
    \mathcal{H}^{t} \left( L_{\theta} \cap \bigcup_{z \in K} \Gamma(z) \right) = 0.
  \end{displaymath}
  Then 
  \begin{equation}\label{form111}
    \Hd \calL_{t} \leq \frac{\Hd K + t}{2\Hd K},
  \end{equation}
  where $\Hd \calL_{t}$ is the Hausdorff dimension of $\{\theta \in [0,2\pi) : L_{\theta} \in \mathcal{L}_{t}\}$. 
\end{thm}

\begin{proof} It is sufficient to show that $\Hd \{\theta \in I : L_{\theta} \in \mathcal{L}_{t}\} \leq \tfrac{1}{2} (\Hd K + t)/\Hd K$ for every "short enough" sub-interval $I \subset [0,2\pi)$ separately. This observation will be used when we apply Lemma \ref{incidenceSineCurves} below: in the statement of the lemma, the set "$\Gamma^{\delta}(z)$" is defined relative to any compact interval $\tfrac{J}{2} \subset [0,2\pi)$, which is sufficiently short that Lemma \ref{Edelta} can be applied.  So, we let $J \subset [0,2\pi)$ be any compact interval such that Lemma \ref{incidenceSineCurves} applies with the definition $\Gamma^{\delta}(z) = \{(\theta,\theta') \in \tfrac{J}{2} \times \R : |\theta' - \rho_{\theta}(z)| \leq \delta\}$, $z \in \R^{3}$.

As a second reduction, we may assume that $K$ is compact: by a result of Davies \cite[Corollary 2]{MR0053184}, the analytic set $K \subset \R^{3}$ contains a compact subset of every dimension strictly smaller than $\Hd K$, and the bound \eqref{form111} is a continuous function of $\Hd K$.

  Fix $t < s < \Hd K$ and use Frostman's lemma to choose a probability measure $\mu$ with $\spt \mu \subset K$, such that $\mu(B(z,r)) \leq \mathbf{C}r^s$ for all balls $B(z,r) \subset \R^{3}$, and for some constant $\mathbf{C} \geq 1$. We make the counter assumption that
  \begin{displaymath}
    \Hd \{\theta \in \tfrac{J}{2} : L_{\theta} \in \mathcal{L}_{t}\} > \alpha > \frac{s + t}{2s},
  \end{displaymath}
  and we choose a Radon probability measure $\sigma$, supported on $\Theta_{t} := \{\theta \in \tfrac{J}{2} : L_{\theta} \in \mathcal{L}_{t}\}$, with $\sigma(B(\theta,r)) \lesssim r^{\alpha}$. The use of Frostman's lemma is legitimate here, since 
  \begin{displaymath} \Theta_{t} = \Big\{\theta \in \tfrac{J}{2} : \mathcal{H}^{t}(L_{\theta} \cap \bigcup_{z \in K} \Gamma(z)) = 0 \Big\} = \bigcap_{\epsilon > 0} \Big\{\theta \in \tfrac{J}{2} : \mathcal{H}^{t}_{\infty}(L_{\theta} \cap \bigcup_{z \in K} \Gamma(z)) < \epsilon\Big\} \end{displaymath}
  is a $G_{\delta}$-set, using the assumption that $K$ is compact (this implies that $L_{\theta} \cap \bigcup_{z \in K} \Gamma(z)$ is also compact for every $\theta \in \tfrac{J}{2}$).
  
By definition of $\calL_{t}$, for every $\theta \in \Theta_{t}$, hence $L_{\theta} \in \calL_{t}$, we may find a collection of arbitrarily short dyadic intervals $\calI_{\theta}$ on $L_{\theta}$, say shorter than $2^{-k_{0}}$, with the following properties:
  \begin{itemize}
    \item[(i)] $L_{\theta} \cap  \bigcup_{z \in K} \Gamma(z) \subset \bigcup_{I \in \calI_{\theta}} I$,
    \item[(ii)] $\sum_{I \in \calI_{\theta}} |I|^{t} \leq 1$.
  \end{itemize}
  The constant $k_{0} \in \N$ will eventually be chosen large in a manner depending only on $\alpha,s,t,\mathbf{C}$. If $\calI_{\theta}' \subset \calI_{\theta}$ is any sub-family, write $\Gamma^{-1}(\calI_{\theta}') \subset \R^{3}$ for all the points $z \in \R^{3}$ such that the point $\Gamma(z) \cap L_{\theta}$ is covered by the intervals in $\calI_{\theta}'$:
  \begin{equation}\label{gamma1}
    \Gamma^{-1}(\calI_{\theta}') := \Big\{z \in \R^{3} : \{\Gamma(z) \cap L_{\theta}\} \subset \bigcup_{I \in \calI_{\theta}'} I \Big\}.
  \end{equation}
  This is a convenient abuse of notation: for instance, now (i) simply states that $\Gamma^{-1}(\calI_{\theta}) \supset K$, and so $\mu(\Gamma^{-1}(\calI_{\theta})) = 1$. For $k \geq 0$, let $\calI_{\theta}^{k}$ be the sub-family of dyadic intervals in $\calI_{\theta}$ with side-length $2^{-k}$, so that $\calI_{\theta} = \bigcup_{k \geq k_{0}} \calI_{\theta}^{k}$. Consequently,
  \begin{displaymath}
    1 = \sigma(\Theta_{t}) = \int_{\Theta_{t}} \mu(\Gamma^{-1}(\calI_{\theta})) \dd\sigma(\theta) \leq \sum_{k \geq k_{0}} \int_{\Theta_{t}} \mu(\Gamma^{-1}(\calI_{\theta}^{k})) \dd\sigma(\theta).
  \end{displaymath}
  It follows that there exists $k \geq k_{0}$ such that
  \begin{equation}\label{form8}
    \int_{\Theta_{t}} \mu(\Gamma^{-1}(\calI_{\theta}^{k})) \dd\sigma(\theta) \gtrsim \frac{1}{k^{2}}.
  \end{equation}
  Write $\delta := 2^{-k}$, so that $k = \log(1/\delta)$. We infer from \eqref{form8} that there exists a subset $\Theta \subset \Theta_{t}$ with $\sigma(\Theta) \gtrsim \log^{-2}(1/\delta)$ such that $\mu(\Gamma^{-1}(\calI_{\theta}^{k})) \gtrsim \log^{-2}(1/\delta)$ for all $\theta \in \Theta$.

 Fix $\theta \in \Theta$. For $j \geq 0$, let $\calI_{\theta}^{k,j}$ consist of those intervals $I \in \calI_{\theta}^{k}$ such that $2^{-j - 1} < \mu(\Gamma^{-1}\{I\}) \leq 2^{-j}$. Then
  \begin{displaymath}
    \log^{-2}(1/\delta) \lesssim \mu(\Gamma^{-1}(\calI_{\theta}^{k})) \leq \sum_{j \geq 0} \mu(\Gamma^{-1}(\calI_{\theta}^{k,j})),
  \end{displaymath}
  so there exists $j = j_{\theta} \geq 0$ such that 
  \begin{equation}\label{form10}
    \mu(\Gamma^{-1}(\calI^{k,j}_{\theta})) \gtrsim \frac{1}{j^{2}\log^{2}(1/\delta)}.
  \end{equation}
  Using (ii), we can estimate
  \begin{displaymath}
    \frac{1}{j^{2}\log^{2}(1/\delta)} \lesssim \mu(\Gamma^{-1}(\calI_{\theta}^{k,j})) \leq \sum_{I \in \calI_{\theta}^{k,j}} \mu(\Gamma^{-1}\{I\}) \leq |\calI_{\theta}^{k}|2^{-j} \leq \delta^{-t}2^{-j},
  \end{displaymath}
  which gives 
  \begin{equation}\label{2j-estimate}
    j^{2}2^{-j} \gtrsim \delta^{t}/\log^{2}(1/\delta).
  \end{equation}
  In particular, this implies that $2^{j} \lesssim \delta^{-1}$, so $j \lesssim \log(1/\delta)$, and we can replace \eqref{form10} and \eqref{2j-estimate} by the slightly tidier estimates
  \begin{equation}\label{form9}
    \mu(\Gamma^{-1}(\calI_{\theta}^{k,j})) \gtrsim \frac{1}{\log^{4}(1/\delta)} \quad \text{and} \quad 2^{-j} \gtrsim \frac{\delta^{t}}{\log^{4}(1/\delta)}.
  \end{equation}
  Now, fix $\eta > 0$ so small that 
  \begin{equation}\label{def:eta}
    0 < \eta < \frac{2s\alpha - s - t}{3s}
  \end{equation}
  (note that the right hand side is positive by the relation between $\alpha,s,t$, and the choice of $\eta$ only depends on these parameters), and apply Lemma \ref{incidenceSineCurves} at scale $5\delta$ with this $\eta > 0$, 
  \begin{equation*}
    \lambda =\delta^{1 - \alpha + \eta}, \quad \text{and} \quad A=C_{\eta,\mathbf{C},s} \cdot \delta^{-\eta},
  \end{equation*}
  where $C_{\eta,\mathbf{C},s} = C_{\alpha,s,t,\mathbf{C}} \geq 1$ is the large constant specified in Lemma \ref{incidenceSineCurves}. The output is a subset $G=G(A,\delta,\lambda) \subset K$ with $\mu(K \setminus G) \leq (C_{\eta,\mathbf{C},s})^{-s/3} \cdot \delta^{\eta s/3}$, and such that 
  \begin{equation}\label{form3}
    |\Gamma^{5\delta}(z) \cap \{w:m^{\mu}_{5\delta}(w) \geq (C_{\eta,\mathbf{C},s})^{s} \delta^{s(2\alpha - 1 - 3\eta)}\}| \leq \lambda |\Gamma^{5\delta}(z)|, \qquad z \in G.
  \end{equation}
  Using the first estimate in \eqref{form9}, we obtain
  \begin{displaymath}
    \frac{1}{\log^{4}(1/\delta)} \lesssim \mu(\Gamma^{-1}(\calI_{\theta}^{k,j})) \leq \mu(\Gamma^{-1}(\mathcal{I}_{\theta}^{k,j}) \cap G) + \mu(K \setminus G), \qquad \theta \in \Theta,
  \end{displaymath}
  which combined with $\mu(K \setminus G) \leq  (C_{\eta,\mathbf{C},s})^{-s/3} \cdot \delta^{\eta s/3}$ gives
  \begin{displaymath}
    \mu(\Gamma^{-1}(\calI_{\theta}^{k,j}) \cap G) \gtrsim \frac{1}{\log^{4}(1/\delta)}, \qquad \theta \in \Theta,
  \end{displaymath}
  for small enough $\delta > 0$. Writing
  \begin{displaymath}
    I^{k,j}_{\theta} := \bigcup_{I \in \calI_{\theta}^{k,j}} I \subset L_{\theta}
  \end{displaymath}
  and recalling that $\sigma(\Theta) \gtrsim \log^{-2}(1/\delta)$, it follows that
  \begin{displaymath}
    \frac{1}{\log^{6}(1/\delta)} \lesssim \int_{\Theta} \mu(\Gamma^{-1}(\calI_{\theta}^{k,j}) \cap G) \dd\sigma(\theta) \leq \int_{G} \sigma(\{\theta \in \Theta : \{\Gamma(z) \cap L_{\theta}\} \subset I^{k,j}_{\theta}\}) \dd\mu(z),
  \end{displaymath}
  which implies the existence of $z_{0} \in G$ with
  \begin{equation}\label{form12}
    \sigma(\{\theta \in \Theta : \{\Gamma(z_{0}) \cap L_{\theta}\} \subset I_{\theta}^{k,j}\}) \gtrsim \frac{1}{\log^{6}(1/\delta)}.
  \end{equation}
  For $\theta \in \Theta$, let $I_{\theta}^{0}  \subset L_{\theta}$ be the unique dyadic $\delta$-interval containing the intersection point $\Gamma(z_{0}) \cap L_{\theta}$; in other words, the estimate \eqref{form12} then says that $I_{\theta}^{0} \in \calI_{\theta}^{k,j}$ for many parameters $\theta \in \Theta$. Let us make this more precise. Since $\sigma(B(x,r)) \lesssim r^{\alpha}$, the lower bound in \eqref{form12} implies that it takes $\gtrapprox \delta^{-\alpha}$ balls of radius $\delta$ to cover the set on the left hand side of \eqref{form12}. In other words, there exist at least $M \gtrapprox \delta^{-\alpha}$ disjoint intervals $I_{1},\ldots,I_{M} \subset \R$ of length $\delta$ such that, for each $1 \leq i \leq M$, the $\delta$-tube $T_{i} := I_{i} \times \R$ contains a segment $I^{0}_{\theta_{i}} \in \calI_{L}^{k,j}$; see Figure \ref{fig3} for illustration.

   \begin{figure}[t!]
  \begin{center}
    \begin{tikzpicture}[scale=0.7]
      \draw (0,7) -- (0,0) -- (1,0) -- (1,7);
      \node[align=center] at (0.5,-0.5) {$I_1$};
      \node[align=center] at (0.5,7) {$T_1$};
      \draw (0.3,3.5) -- (0.3,5);
      \node[align=center] at (1.8,5) {$I^0_{\theta_{1}}$};
      \draw [->] (1.3,5) arc (100:135:1.5);

      \draw (3,7) -- (3,0) -- (4,0) -- (4,7);
      \node[align=center] at (3.5,-0.5) {$I_2$};
      \node[align=center] at (3.5,7) {$T_2$};
      \draw (3.6,2.5) -- (3.6,4);
      \node[align=center] at (5.1,4) {$I^0_{\theta_{2}}$};
      \draw [->] (4.6,4) arc (100:135:1.5);

      \draw (7,7) -- (7,0) -- (8,0) -- (8,7);
      \node[align=center] at (7.5,-0.5) {$I_3$};
      \node[align=center] at (7.5,7) {$T_3$};
      \draw (7.7,4) -- (7.7,5.5);
      \node[align=center] at (6.2,5.5) {$I^0_{\theta_{3}}$};
      \draw [->] (6.7,5.5) arc (80:45:1.5);

      \draw (9,7) -- (9,0) -- (10,0) -- (10,7);
      \node[align=center] at (9.5,-0.5) {$I_4$};
      \node[align=center] at (9.5,7) {$T_4$};
      \draw (9.4,1.5) -- (9.4,3);
      \node[align=center] at (10.9,3) {$I^0_{\theta_{4}}$};
      \draw [->] (10.4,3) arc (100:135:1.5);

      \node[align=center] at (12.5,3.5) {$\cdots$};

      \draw (15,7) -- (15,0) -- (16,0) -- (16,7);
      \node[align=center] at (15.5,-0.5) {$I_M$};
      \node[align=center] at (15.5,7) {$T_M$};
      \draw (15.5,4) -- (15.5,5.5);
      \node[align=center] at (13.9,5.5) {$I^0_{\theta_{M}}$};
      \draw [->] (14.5,5.5) arc (80:45:1.5);
    \end{tikzpicture}
    \caption{An illustration for the proof of Theorem \ref{mainint}.}
    \label{fig3}
  \end{center}
  \end{figure}
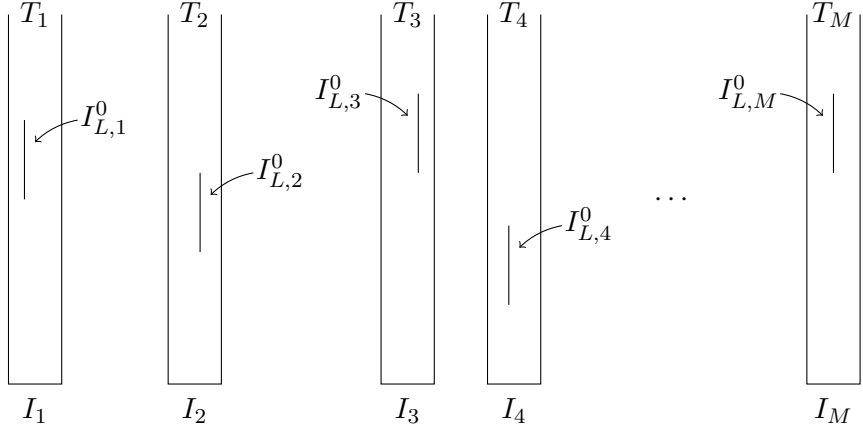

  Finally, recall that
  \begin{displaymath}
    m^{\mu}_{5\delta}(w) := \mu(\{z' \in \R^{3} : w \in \Gamma^{5\delta}(z')\}).
  \end{displaymath}
  A basic observation is the following: if $\theta \in \Theta \subset \tfrac{J}{2}$ and $I \subset L_{\theta}$ is a vertical segment of length $\delta$ (in particular $I_{\theta_{i}}^{0}$ for some $i$), and 
  \begin{displaymath} w \in I(\delta) := \{w' \in \R^{2} : \dist(w',I) \leq \delta\}, \end{displaymath}
  then $w \in \Gamma^{5\delta}(z)$ for all $z \in \Gamma^{-1}\{I\}$. Indeed, if $z \in \Gamma^{-1}\{I\}$, then $\{ \Gamma(z) \cap L_{\theta}\} \subset I$. Moreover, $\{ \Gamma(z) \cap L_{\theta}\} = (\theta, \rho_\theta(z))$ for some $\theta \in \tfrac{J}{2}$. Thus, writing $w=(w_{1},w_{2}) \in I(\delta)$, we have $|w_{1} - \theta| \leq \delta$ and $|w_{2} - \rho_{\theta}(z)| \leq 2\delta$, and hence
  \begin{align*}
    |w_{2} -\rho_{w_{1}}(z)| \le |w_{2} - \rho_{\theta}(z)| + |\rho_{\theta}(z) - \rho_{w_{1}}(z)| \leq 5\delta.
  \end{align*}
  As a consequence,
  \begin{displaymath}
    m^{\mu}_{5\delta}(w) \geq \mu(\Gamma^{-1}\{I\}), \qquad w \in I(\delta),
  \end{displaymath}
  and in particular
  \begin{equation}\label{form11}
    m_{i} := |T_{i} \cap \Gamma^{\delta}(z_{0}) \cap \{w : m^{\mu}_{5\delta}(w) \geq \mu(\Gamma^{-1}\{I_{\theta_{i}}^{0}\})\}| \geq |T_{i} \cap \Gamma^{\delta}(z_{0}) \cap I_{\theta_{i}}^{0}(\delta)| \sim \delta^{2}.
  \end{equation}
  Next, recall that
  \begin{displaymath}
    \mu(\Gamma^{-1}\{I_{\theta_{i}}^{0}\}) \sim 2^{-j} \gtrsim \frac{\delta^{t}}{\log^{4}(1/\delta)}, \qquad 1 \leq i \leq M
  \end{displaymath}
  by the second estimate in \eqref{form9}, since $I_{\theta_{i}}^{0} \in \calI_{\theta}^{k,j}$. If $\delta > 0$ is sufficiently small, depending on $\alpha,s,t,\mathbf{C}$ (this can be arranged by choosing $k_{0} \in \N$ large enough to begin with, and recalling that $\delta = 2^{-k} \leq 2^{k_{0}}$) the right hand side exceeds $C_{\eta,\mathbf{C},s}^{s} \delta^{s(2\alpha - 1 - 3\eta)}$, by the choice of $\eta$, recall \eqref{def:eta}. By \eqref{form11} and the disjointness of the vertical tubes $T_{i}$, this means that
  \begin{displaymath}
    |\Gamma^{\delta}(z_{0}) \cap \{w : m_{5\delta}^{\mu}(w) \geq C^{s}_{\eta,\mathbf{C},s} \delta^{s(2\alpha - 1 - 3\eta)}\}| \geq \sum_{i = 1}^{M} m_{i} \gtrapprox \delta^{2 - \alpha} \sim \delta^{-\eta} \lambda |\Gamma^{5\delta}(z_{0})|.
  \end{displaymath} 
  Since $\eta > 0$ and $z_{0} \in G$, this contradicts \eqref{form3} for sufficiently small $\delta > 0$. The proof is complete.
\end{proof}

With the same argument we can also prove the following lemma about circles, which will then imply Theorem \ref{circleUnion}.

\begin{lemma}\label{circleSlices}
  Let $K \subset \mathbf{B}_{0}$ be an analytic set and let $\calL$ be a set of vertical lines $L_{\theta} = \{\theta\} \times \R$ with $-\tfrac{1}{4} \le \theta \le \tfrac{1}{4}$ such that
  \begin{displaymath}
    \Hd \left( L_{\theta} \cap \bigcup_{z \in K} S(z) \right) < \min \{ \Hd K,1\}.
  \end{displaymath}
  Then $\{\theta \in [-\tfrac{1}{4},\tfrac{1}{4}] : L_{\theta} \in \mathcal{L}\} = 0$. 
\end{lemma}

\begin{proof}
  We may assume that $0 < \Hd K \leq 1$. Fix $0 < t < s < \Hd K$, and pick a probability measure $\mu$ with $\spt \mu \subset K$ and $\mu(B(z,r)) \lesssim r^{s}$. The previous proof can be used to show that $\Hd \calL_{t} \leq (\Hd K + t)/(2\Hd K) < 1$, where $\calL_{t} \subset \calL$ is the collection of those lines $L_{\theta}$ with 
  \begin{displaymath}
    \calH^{t}\left(L_{\theta} \cap \bigcup_{z \in K} S_{+}(z) \right) = 0,
  \end{displaymath}
  and $S_{+}(z)$ is the upper half of the circle $S(z)$. Lemma \ref{circleSlices} is evidently a corollary of this statement, so we only need to indicate the proof of that statement. First note that since we consider only those vertical lines $L_{\theta}$ with $-1/4\le \theta \le 1/4$, they intersect every half-circle $S_{+}(z)$ with $z \in K$ exactly once. This is due to the fact that $K \subset \mathbf{B}_{0}$, thus the centre of any circle $S(z)$ lies in $B(0,\tfrac{1}{2})$, and the radius is at least $1/2$.

  In analogy with the proof of Theorem \ref{mainint}, we can define $S^{-1}_+(\calI'_L)$ for any family of intervals $\calI'_L$ as was done in \eqref{gamma1} for $\Gamma^{-1}(\calI'_L)$. Instead of Lemma \ref{incidenceSineCurves}, we now use its corresponding version for circles, Lemma \ref{incidenceCircles}. As we are using half circles, we need to modify the multiplicity function as well, so instead of $m^\mu_\delta$, which was defined for circles in \eqref{multcircles}, we define it for half circles as
  \begin{displaymath}
    m^\mu_{\delta,+}(w)=\{ z' \in \R^3: w \in S^\delta_+(z')\},
  \end{displaymath}
  where $S^\delta_+(z)$ is the $\delta$ neighbourhood of $S_+(z)$. Since $m^\mu_{\delta,+}(w) \le m^\mu_\delta(w)$ for every $w \in \R^2$, it follows that the conclusion of Lemma \ref{incidenceCircles} holds still true when $m^\mu_\delta$ is replaced by $m^\mu_{\delta,+}$. In particular, with the same parameters $A,\alpha,\eta,\mathbf{C},\lambda,s,t$ as in \eqref{form3}, we can find a subset $G \subset K$ with $\mu(K \setminus G) \le C_{\eta,\mathbf{C},s}^{-s/3}\delta^{\eta s/3}$ such that for every $z \in G$,
  \begin{displaymath}
    |S^{5\delta}_+(z) \cap \{w: m^\mu_{5\delta,+}(w) \ge C_{\eta}^{s} \delta^{s(2\alpha - 1 - 3\eta)} \}| \le \lambda |S^{5\delta}_+(z)|.
  \end{displaymath}
  From this point on, the proof is exactly the same as that of Theorem \ref{mainint}. Note that the "basic observation" between \eqref{form12} and \eqref{form11} is still valid: if $L \in \calL$ and $I \subset L$ is a vertical segment of length $\delta$ and $w \in I(\delta)$, then $w \in S_+^{5\delta}(z)$ for every $z \in S^{-1}_+\{I\}$. Indeed, for every $z \in S^{-1}_+\{I\}$ we have $ \{S_+(z) \cap L\} \subset I$, that is $|w- \{S_+(z) \cap L\}| \le 5\delta$, which implies $w \in S^{5\delta}_+(z)$.
\end{proof}

Lemma \ref{circleSlices} implies Theorem \ref{circleUnion}, which we restate here. Recall that $\Hd \calS := \Hd \{z \in \R^{3} : S(z) \in \calS\}$ and $\cup\calS = \bigcup_{S \in \calS} S$.

\begin{thm}
  If $\calS$ is an analytic family of circles, then $\Hd \cup \calS = \min\{\Hd \calS+1,2\}$.
\end{thm}

\begin{proof}
  Fix $0 \leq t < \min\{\Hd \calS,1\}$. By Lemma \ref{circleSlices}, for almost every $\theta \in [-\tfrac{1}{4},\tfrac{1}{4}]$, the vertical line $L_{\theta} = \{\theta\} \times \R$ satisfies
  \begin{displaymath}
    \Hd [\cup \calS \cap L_{\theta}] \geq t.
  \end{displaymath}
  Hence, by \cite[Theorem 5.8]{Fa}, we have $\Hd \cup \calS \geq t + 1$, and the lower bound of the theorem now follows by letting $t \uparrow \min\{\Hd \calS,1\}$. The upper bound follows by a standard covering argument, and we omit the details.
\end{proof}

\bibliographystyle{plain}
\bibliography{references}

\end{document}